\documentclass[10pt, english, oneside, a4paper]{article}

\usepackage[utf8]{inputenc}
\usepackage{bm}
\usepackage[style=american]{csquotes}
\usepackage{rotating}
\usepackage[hidelinks]{hyperref}
\usepackage[left=1.25in,right=1.25in,top=1.25in,bottom=1.25in]{geometry} 
\usepackage{etoolbox}
\usepackage{tcolorbox}
\usepackage{enumitem}
\usepackage{scrextend}
\usepackage{centernot}
\usepackage{lipsum}
\usepackage{amsfonts}
\usepackage{amssymb}
\usepackage{amsmath}
\usepackage{amsthm}
\usepackage{amsopn}
\usepackage{graphicx}
\usepackage{epstopdf}
\usepackage{benstyle_arXiv}
\usepackage{tikz}
\usepackage{bbm}
\usepackage{mathtools}
\usepackage{array}
\usepackage{rotating}
\usepackage{booktabs}
\usepackage[normalem]{ulem}

\usepackage{graphicx}
\usepackage{varioref, babel, fancyvrb, listings, amsmath,amsthm,  mathtools,bbm}
\usepackage{amssymb}
\usepackage{dsfont}
\usepackage{enumitem}
\usepackage{array}
\usepackage{booktabs}
\let\oldcap\cap
\usepackage{mathabx}
\let\cap\oldcap
\usepackage{cite}
\usepackage{times}
\usepackage{listings}
\usepackage[normalem]{ulem}
\usepackage[style=american]{csquotes}
\usepackage{cleveref}

\usepackage{lipsum}
\usepackage{amsfonts}
\usepackage{graphicx}
\usepackage{epstopdf}

\usepackage{algorithm}
\usepackage{algpseudocode}

\usepackage[dvipsnames]{xcolor}
\usepackage{pgfplots}
\pgfplotsset{compat=1.18}
\usepackage{subcaption}
\usepackage{tikz}

\usepackage{bm}
\usepackage{dsfont}

\usepackage{caption}
\usepackage{amssymb}
\newcommand{\cupdot}{\mathbin{\mathaccent\cdot\cup}}

\usetikzlibrary{calc}
\usetikzlibrary{decorations.pathreplacing,calc}

\ifpdf
  \DeclareGraphicsExtensions{.eps,.pdf,.png,.jpg}
\else
  \DeclareGraphicsExtensions{.eps}
\fi

\numberwithin{equation}{section}

\theoremstyle{plain}
\newtheorem{theorem}{Theorem}[section]
\newtheorem{lemma}[theorem]{Lemma}

\newtheorem{corollary}[theorem]{Corollary}
\newtheorem{mainxx}[theorem]{Main result}

\theoremstyle{definition}
\newtheorem{definition}[theorem]{Definition}

\theoremstyle{remark}
\newtheorem{remark}[theorem]{Remark}

\usepackage{enumitem}
\usepackage[percent]{overpic}
\usepackage{pict2e}

\usepackage{amsopn}
\usepackage{cite}

\DeclareMathOperator*{\argmin}{argmin}

\AtBeginEnvironment{figure}{\setlength{\tabcolsep}{2pt}}

\tikzstyle{stuff_fill}=[rectangle,draw,fill=pink,minimum size=0.5em]
\usetikzlibrary{calc}
\usetikzlibrary{decorations.pathreplacing,calc}
\newcounter{arrow}
\title{Average Kernel Sizes \\-\\ Computable Sharp Accuracy Bounds for Inverse Problems\thanks{
NMG, DI and JG acknowledge the funding provided by the German Aerospace Center from the department EO Data Science at the Earth Observation Center. This work was supported by the Helmholtz Association's Initiative and Networking Fund on the HAICORE (at) FZJ partition.
}}

\author{Nina M.~Gottschling
  \thanks{German Aerospace Center (DLR), Remote Sensing Technology Institute, Wessling, Germany.}
\and David Iagaru
    \footnotemark[2]
    \thanks{Department of Applied Mathematics, École Polytechnique, Paris, France.}
  \and Jakob Gawlikowski
    \footnotemark[2]
\and Ioannis Sgouralis
  \thanks{Department of Mathematics, University of Tennessee, Knoxville, Tennessee.}}

\usepackage{amsopn}

\ifpdf
\hypersetup{
  pdftitle={Average Kernel Sizes},
  pdfauthor={Nina M. Gottschling, David Iagaru, Jakob Gawlikowski and Ioannis Sgouralis}
}
\fi

\makeatletter
\renewenvironment*{displayquote}
  {\begingroup\setlength{\leftmargini}{0.6cm}\csq@getcargs{\csq@bdquote{}{}}}
  {\csq@edquote\endgroup}
\makeatother

\begin{document}

\maketitle

\begin{abstract}
Reconstructing unknown quantities from noisy measurements is a central challenge in applied sciences, including medical imaging, radar, and astronomy. These inverse problems are often ill-posed, implying a fundamental tradeoff between reconstruction accuracy and stability [Gottschling et al., \textit{SIAM Review}, 67.1 (2025), pp. 73-104] and [Colbrook et al., \textit{Proceedings of the National Academy of Sciences}, 119.12 (2022), pp. e2107151119]. Consequently, all stable approximate inverse maps exhibit nonzero accuracy limits for ill-posed inverse problems. While many approaches have been developed to design approximate inverse maps, ranging from optimization-based and Bayesian methods to deep learning, there are currently no computable method-independent bounds on the optimal achievable reconstruction accuracy. In this work, we derive computable sharp accuracy bounds for inverse problems that depend only on the dataset of signals, the forward model, and the noise model. Computing these bounds requires the ability to generate or approximate sets of solutions and corresponding measurements. Our framework enables estimating achievable reconstruction accuracy before designing or optimizing inverse maps. 
Under additional measurability and sampling assumptions, we show that the computable bounds converge to their measure-theoretic counterparts and that the lower bound, one half of the average kernel size, forms an asymptotic probabilistic accuracy barrier: no approximate inverse map can asymptotically achieve a lower reconstruction error while maintaining convergence in probability of empirical to generalization error. To facilitate practical use, we provide an algorithmic framework and software library for computing the bounds, and demonstrate the approach on fluorescence localization microscopy and multi-spectral satellite image super-resolution. The proposed bounds provide both theoretical and empirical insight into the fundamental accuracy limits of inverse problems and support the principled design of datasets and forward models.
\end{abstract}

\paragraph*{Keywords}
Inverse problems, optimization, accuracy bounds, uncertainty, imaging, deep learning 

\section{Introduction}

In recent years, the intersection of artificial intelligence and inverse problems has revolutionized various scientific fields. 
Inverse problems, at their core, involve reconstructing an unknown signal, such as an image or features of an image, from indirect or incomplete measurements. This task is often ill-posed and computationally challenging \cite{bertero2021introduction}. Ill-posed finite dimensional inverse problems are ubiquitous in the computational sciences as they appear in multiple applications. An incomplete list of examples includes most types of computational imaging \cite{adcock2021compressive, bouman2022foundations, barrett2013foundations}, matrix completion \cite{candes2012exact}, parametric partial differential equations and system identification \cite{adcock2022sparse}, phase retrieval \cite{shechtman2015phase, candes2013phaselift, fannjiang2020numerics}, and quantized sampling \cite{4558487}. Examples of applications are radar inverse scattering \cite{borden2001mathematical} or astronomy \cite{craig1986inverse}. Traditional reconstruction methods, such as regularization inverse techniques \cite{tikhonov1963solution} or compressed sensing \cite{FoucartRauhutCSbook}, have achieved significant success. However, emerging AI and deep learning workflows have ushered in a new era of methods. AI-based methods and combinations of optimization-based iterative and AI methods, such as \cite{colbrook2022difficulty}, have enabled remarkable improvements in reconstruction quality and speed in various inverse problems. These include, but are not limited to, the inverse problems of Magnetic Resonance Imaging (MRI) \cite{fastmri19,Bo-18}, microscopy \cite{hoff21}, and computed tomography (CT) \cite{wang2020deep}. For example, deep neural networks, trained on large datasets, can learn complex data-driven reconstruction methods that outperform classical methods in MRI \cite{hammernik2018learning}. 

\subsection{What are accuracy limits in benchmark challenges?\nopunct}

Despite the expressive power and universal approximation properties of deep neural networks \cite{hornik1989multilayer,maiorov1999lower,bach2017breaking,beck2022full,petersen2018optimal,yarotsky2017error}, state-of-the-art methods for inverse problems often achieve remarkably similar reconstruction errors on benchmark datasets. Such accuracy plateaus have been observed across a range of inverse problems, including MRI reconstruction, tomography, phase retrieval, inverse scattering, and scientific imaging benchmarks \cite{fastmri19,fastmri20,fastmridata,crafts2025benchmarking,zheng2025inversebench}. This raises the question: what are the intrinsic dataset- and problem-dependent accuracy limits for inverse problems? In theory, deep neural networks are universal approximators that obtain vanishing errors. However, as the authors of \cite{adcock2021gap} point out, there is a gap between theoretical guarantees and the practical training and design of such networks. This observed gap between theory and practice even further motivates the question of what inverse problem- and dataset-dependent accuracy limits are. Existing theoretical results establish that stable reconstruction methods on measurable metric spaces necessarily exhibit nonzero reconstruction error for ill-posed inverse problems \cite{plaskota1996noisy,gottschling2023existence}. However, it remains open whether such limits admit computable finite-sample bounds that have favourable convergence properties for general datasets  \cite{burger2024learning}. The present work addresses this question by deriving computable method-independent accuracy bounds from finite datasets and inverse problems. Under additional measurability and sampling assumptions, we show that no approximate inverse map can simultaneously achieve reconstruction error below this threshold and obtain convergence in probability of the empirical to the generalization error. 

\section{Overview of the paper}

In this section, we present a brief overview of the paper. We formalize the problem studied, introduce three key challenges in the related literature, and give a summary of our main result.

\subsection{Problem outline}

\noindent We base our problem outline on the general pipeline of solving inverse problems using data-driven models \cite{Arr19}, where a dataset for training, testing and validation is available. Following \cite{gottschling2023existence}, we express an inverse problem using the following general framework:
\begin{align}\label{eq:sampling10}
	\text{recover } x \in \mathcal{M}_1 \subset \mathcal{X} \text{ given a point } y = F(x,e)\in \mathcal{Y} \text{ of } x  \text{ and }  e \in \mathcal{E}\subset \mathcal{Z}.
\end{align}
Here, $x$ represents the signal of interest, while $e$ represents the noise. The sets \noindent $\mathcal{M}_1$ and $\mathcal{E}$ describe the signals of interest and noise, respectively, and may be finite, but we do not assume this. The function $F \colon \mathcal{M}_1\times \mathcal{E} \to \mathcal{M}_2 \subset \mathcal{Y}$ which describes the forward model such that $\mathcal{M}_2 = F(\mathcal{M}_1 \times \mathcal{E})$ and is deliberately kept general to encompass many known models. The aim of solving the inverse problem is to produce an approximation or set of approximations $\hat{x}$ to the true solution $x$. Any algorithm or method that yields an approximate inverse map from noisy measurements in the set $\mathcal{M}_2$ to an approximation to the true solution $x$, is denoted by $\phi: \mathcal{M}_2 \rightarrow \mathcal{X}$ with $\phi(y) = \hat{x}$. Some examples are approximate inverse maps obtained from optimization-based inverse methods, such as compressed sensing \cite{FoucartRauhutCSbook} or total variation minimization \cite{rudin1992nonlinear,huang2009new}, combinations of optimization and data-driven inverse methods \cite{Arr19,borcea2023waveform}, and (un-)supervised Deep Learning methods \cite{Str-18,bel19,Bo-18}. A common choice for optimization-based approximate inverse map is to add a weighted regularization term $\alpha\mathcal{R}(\cdot)$ to the optimization objective, which often takes the following form
\begin{align}\label{eq:regsolvers}
\phi_{\alpha}(y) = \hat{x} \in \argmin_{z \in \mathcal{X}}\left(d_{\mathcal{Y}}(y,F(z,0))+\alpha \mathcal{R}(z)\right), \quad \text{ for } \alpha \in (0,\infty),
\end{align}
where $d_{\mathcal{Y}}$ is a metric on $\mathcal{Y}$, that is often induced by a norm \cite{Arr19}. A comprehensive mathematical review of data-driven methods for solving inverse problems is provided in \cite{Arr19}. The reconstruction error of an approximate inverse map $\phi: \mathcal{M}_2 \rightarrow \mathcal{X}$ for approximately solving \eqref{eq:sampling10} is often quantified by the empirical expectation of the residual of the reconstruction and solutions of the $M \in \mathbb{N}$ samples $\mathcal{D}_M=\{(x_m,y_m)\}_{m=1}^M = \{(x_m,F(x_m,e_m))\}_{m=1}^M \subset \mathcal{M}_1 \times \mathcal{M}_2$, and can be denoted by
\begin{align}\label{eq:empiricalerror}
    \mathcal{L}(\mathcal{D}_M, \phi,p) = \left(\frac{1}{M} \sum_{m=1}^M d_{\mathcal{X}}(x_m,\phi(y_m))^p\right)^{1/p},
\end{align}
\noindent for some $p \in [1,\infty)$ and some (pseudo) metric $d_{\mathcal{X}}$ on $\mathcal{X}$. For $p=1$ this is the mean absolute error, and for $p=2$ this is the root mean squared error (RMSE) and is also known as empirical risk or root mean squared loss, which are commonly used to evaluate the accuracy of models and approximate inverse maps, see, e.g., [\cite{bickel2015mathematical}, Chapter 1] and [\cite{murphy2022probabilistic}, Chapter 1]. A common objective in inverse problems for applications in various fields, is to find an approximate inverse map that minimizes the loss \eqref{eq:empiricalerror} in a given class of functions, such as finding the optimal architecture and parameter values for deep neural networks \cite{cgenzel2022solving,Arr19}. The \textit{accuracy limit} $c_{\mathrm{lim}}(\mathcal{D}_M,p)$ as a lower bound to the reconstruction error of any approximate inverse map $\phi:\mathcal{M}_2\rightarrow \mathcal{X}$ satisfies 
\begin{align}\label{eq:empiricalerrorlimit}
   c_{\mathrm{lim}}(\mathcal{D}_M,p):=\inf_{\theta} \mathcal{L}(\mathcal{D}_M, \theta,p) \leq \mathcal{L}(\mathcal{D}_M, \phi,p),
\end{align}
\noindent where we take the infimum over all functions $\theta :\mathcal{M}_2\rightarrow \mathcal{X}$ and $c_{\mathrm{lim}}(\mathcal{D}_M,p)$ is non-zero for datasets $\mathcal{D}_M$, such that $(x_{m'},F(x_{m'},e_{m'})),(x_m,F(x_m,e_m)) \in \mathcal{D}_M$ with $x_m\neq x_{m'}$ and $F(x_{m},e_{m})=F(x_{m'},e_{m'})$. Such functions are kept as general as possible and do not necessarily need to be continuous. The accuracy limit in \eqref{eq:empiricalerrorlimit} in theory can be obtained by devising and computing infinitely many approximate inverse maps to solve \eqref{eq:sampling10}. As this is impossible in practice, \emph{the aim in this work is to find a data-driven bound for \eqref{eq:empiricalerrorlimit} that is independent of any special function class.} 

\subsection{Challenges for computing solutions to inverse problems}

\textit{(i) No unique solution:} In real-world scenarios, the measurements $y \in \mathcal{M}_2$ in \eqref{eq:sampling10} may only be partially available and corrupted by errors due to equipment constraints \cite{burger2024learning}. As a result, the same measurement can result from more than one unique signal $x \in \mathcal{M}_1$ and an instance of noise $e \in \mathcal{E}$. In many instances of inverse problems of interest given a measurement, there is no unique solution. For example, this occurs if the problem of recovering $x$ given $y$ in \eqref{eq:sampling10} is ill-posed or undersampled, unless further assumptions are made, such as the robust null space property or the restricted isometry property \cite{CoDaDe-08,FoucartRauhutCSbook}. If the forward model $F$ is poorly conditioned or dimensionality reducing, then, given a measurement, there may not exist a unique solution. This means that for some $y \in \mathcal{M}_2$ the first component of the pre-image of the measurement $y$ under $F$ contains more than one element, 
\begin{align}\label{eq:centralassumption}
    \exists x,x' \in \pi_1(F^{-1}(y))= \{x \in \mathcal{M}_1: \exists e \in \mathcal{E} \text{ s.t } F(x,e) = y\}, \text{ such that } x\neq x',
\end{align}
\noindent where $\pi_1: \mathcal{M}_1 \times \mathcal{E} \rightarrow \mathcal{M}_1$ is the projection onto the first component $\pi_1(x,e) =x$ for $(x,e) \in \mathcal{M}_1\times \mathcal{E}$. \eqref{eq:centralassumption} can be true if $\mathcal{M}_1 \subset \mathbb{C}^{d_1}$ is countable, countably finite, infinite, or uncountable. \newline


\noindent \textit{(ii) Existence of accuracy bounds:} For inverse problems where there are multiple solutions for one measurement, as in \eqref{eq:centralassumption}, we can only obtain approximate inverse maps $\phi: \mathcal{M}_2 \rightarrow \mathcal{X}$. This is equivalent to only being able to obtain a non-zero reconstruction error for inverse maps. In \cite{gottschling2025troublesome} we establish an accuracy-stability trade-off for approximate inverse maps that solve ill-posed inverse problems. In \cite{gottschling2025troublesome,colbrook2022difficulty} we show that approximate inverse maps with smaller error and, thus, higher accuracy are prone to instabilities and incorrect transfer of details. This accuracy-stability trade-off implies that there exists a constant $K(F,\mathcal{M}_1,\mathcal{E},p)>0$ depending on $F$, $\mathcal{M}_1$, $\mathcal{E}$ and $p$ that is independent of any approximate inverse map $\phi: \mathcal{M}_2 \rightarrow \mathbb{C}^{d_1}$ such that
\begin{align}\label{eq:lowerboundlip}
K(F,\mathcal{M}_1,\mathcal{E},p) \leq \mathcal{L}(\mathcal{D}_M,\phi,p),
\end{align}
\noindent for datasets $\mathcal{D}_M$, such that $(x_{m'},y_m),(x_m,y_m) \in \mathcal{D}_M$ with $x_m\neq x_{m'}$, and all stable, Lipschitz continuous, $\phi: \mathcal{M}_2 \rightarrow \mathcal{X}$, where Lipschitz continuity is well-defined on metric spaces. For optimal inverse maps which are measurable, previous works have provided upper and lower bounds to the average and worst-case reconstruction error \cite{gottschling2023existence,plaskota1996noisy}. By providing quantifiable bounds to the reconstruction error of optimal inverse maps, this answers the question posed by M.~Burger and T.~Roith in \textit{Learning in Image Reconstruction: A Cautionary Tale} \cite{burger2024learning},

\begin{displayquote}
   \textit{``Can we ever quantify the reconstruction error [of ill-posed inverse problems]?''}
\end{displayquote}

\noindent However, the provided average bounds in \cite{gottschling2023existence,plaskota1996noisy} are not computable and Monte Carlo approximations may not preserve the inequalities and this is challenging in applications.\newline

\noindent \textit{(iii) Universal approximators do not evade accuracy bounds:} In theory, neural networks are universal approximators and have shown great success in obtaining a vanishing approximation error \cite{hornik1989multilayer,maiorov1999lower,bach2017breaking,beck2022full,petersen2018optimal,yarotsky2017error}. As the authors of \cite{adcock2021gap} point out, such approximation results typically do not provide theoretical guarantees on trainability. In a benchmark test in different realistic inverse problems \cite{zheng2025inversebench}, an accuracy-efficiency trade-off for diffusion models as approximate inverse maps is established. Accuracy bounds to the best worst-case reconstruction error are established by some existing frameworks. For example, for noiseless linear inverse problems, there are the Gelfand widths \cite{pinkus2012n} and the best $k$-term approximation \cite{CoDaDe-08}. Worst-case accuracy bounds for set-valued decoders on Banach spaces have been derived in \cite{arestov1986optimal} and then extended to metric spaces in \cite{magaril1991optimal}. For inverse problems with noise, optimal learning \cite{binev2022optimal} can provide such worst-case accuracy bounds. For linear inverse problems, the notion of generalized instance optimality \cite{fundamental14} provides point-wise upper bounds on the reconstruction error. Worst- and average- case accuracy bounds on measurable metric spaces are established in \cite{gottschling2023existence,plaskota1996noisy}.

\subsection{Contributions}

In this work, we address the challenges that were previously mentioned for computing solutions to inverse problems. Challenge \textit{$(i)$ - that of inverse problems \eqref{eq:sampling10} that have no unique solution -} is a necessary assumption for our main result in \eqref{eq:centralassumption}, that is satisfied if the forward model of \eqref{eq:sampling10} is not injective on the product set of signals and noise $\mathcal{M}_1\times \mathcal{E}$. Due to the non-injectivity of the forward model, there exist measurements $y \in \mathcal{M}_2$ such that the inverse problem \eqref{eq:sampling10} does not have a unique solution. From this the challenge \textit{$(ii)$ - that there exist accuracy bounds -} follows. This leads to the challenge \textit{$(iii)$ - that universal approximators do not evade the accuracy bounds}. To address this, we develop a framework to compute accuracy bounds for the reconstruction error of inverse problems. We provide computable lower and upper accuracy bounds to the empirical reconstruction error in \eqref{eq:empiricalerror} of any inverse map solving an inverse problem \eqref{eq:sampling10}. We now present a nontechnical summary of our main result and refer to Section \ref{sec:main} for the formal statement. The accuracy bounds for any approximate inverse map in \eqref{eq:lowerboundlip} can be computed with the average kernel size. Informally, the average kernel size is the average distance between samples of the $x$-components of  $(x,e), (x',e') \in\mathcal{M}_1\times \mathcal{E}$ that the forward model maps to the same measurement $y=F(x,e) = F(x',e')$ over all the measurement samples. 

\begin{mainxx}[{Accuracy Bounds for Computing Solutions to Inverse Problems - Theorem \ref{thm:lowerboundapprx}}]
Let $p \in [1,\infty)$, then for an inverse problem as in \eqref{eq:sampling10}, we have that half the average kernel size is a sharp lower bound to the reconstruction error of any approximate inverse map. Additionally, the average kernel size is a sharp upper bound to the reconstruction error of approximate inverse maps that attain the accuracy limit in \eqref{eq:empiricalerrorlimit}.
\end{mainxx}

\noindent The average kernel size in Theorem \ref{thm:lowerboundapprx} provides sharp accuracy bounds that depend only on samples from the sets $\mathcal{M}_1$ and $\mathcal{E}$, the forward model $F$ and the evaluation (pseudo) metric for the reconstruction error in \eqref{eq:empiricalerror}. For linear $F$ the bounds depend on its kernel, and the kernel size generalizes the concept of kernel to non-linear forward models. In Section \ref{sec:applications} to demonstrate the applicability of our results and accompanying software library, we apply the framework to two examples of inverse problems with non-injective forward models. As a characteristic example for an inverse problem with a non-injective and highly non-linear forward model, we use fluorescence localization microscopy. For a large-scale inverse problem with a linear and non-injective forward model, we take super-resolution of multi-spectral satellite data, which is a common large scale problem in geoscience \cite{quattrochi1997scale}. 

\noindent The novelty of Theorem \ref{thm:lowerboundapprx} lies in establishing a finite-sample computable lower bound that preserves the inequality exactly. In contrast to the average-case bounds in \cite[Theorem 3.9]{gottschling2023existence}, Theorem \ref{thm:lowerboundapprx} is a geometric-analytic statement that applies to arbitrary approximate inverse maps under non-injectivity of the forward model and does not rely on measurability assumptions or i.i.d. Monte Carlo sampling. By contrast, numerical approximation of the average kernel size in \cite[Definition 3.8]{gottschling2023existence} does not in general preserve the corresponding reconstruction inequality at finite sample level due to sampling and approximation errors. Theorem \ref{thm:lowerboundapprx} therefore complements the measure-theoretic framework of \cite{gottschling2023existence}: whereas \cite{gottschling2023existence} establishes asymptotic bounds on measurable metric spaces, this work provides computable finite-sample bounds that remain valid independently of any probabilistic sampling model. In Section \ref{sec:additionalassconv}, we impose additional measurability, compactness, and sampling assumptions to connect the finite-sample framework of Theorem \ref{thm:lowerboundapprx} with the asymptotic measure-theoretic setting of \cite[Theorem 3.9]{gottschling2023existence}. Theorem \ref{thm:convavgalgnon} shows that the computable average kernel size converges in probability to the corresponding average kernel size on the underlying measure space as the number of i.i.d. samples increases. Furthermore, Lemma \ref{lem:avererrconv} shows that the empirical reconstruction error of measurable approximate inverse maps converges in probability to the corresponding generalization error. Together with the sharp lower and upper reconstruction bounds from Theorem \ref{thm:lowerboundapprx} and \cite[Theorem 3.9]{gottschling2023existence}, these convergence results imply the following corollary.

\begin{mainxx}[{Asymptotically Necessary Condition for Convergence in Probability of the Empirical Error to the Generalization Error - Corollary \ref{thm:finitesampleconsistency}}]
Assume that the empirical reconstruction error converges in probability to the corresponding generalization error as the number of samples increases. Then the finite-sample lower bound from Theorem \ref{thm:lowerboundapprx} becomes an asymptotically necessary condition for this convergence: for every sufficiently small $\varepsilon>0$, the probability that the empirical reconstruction error falls below one half of the computable average kernel size minus $\varepsilon$ converges to zero as the sample sizes tend to infinity.
\end{mainxx}

\subsection{Outlook and impact}

The lower bound to the reconstruction error can be computed before designing and engineering approximate inverse maps, such as training deep neural networks or optimization methods. This is because the bounds only depend on the forward model $F$, and the sets $\mathcal{M}_1$ and $\mathcal{E}$, together with the ability to sample or approximate sets of solutions for selected measurements, i.e. by using simulation-based generation, posterior sampling, or analytic shortcuts in structured settings. The provided accuracy bounds can give researchers and practitioners insight into model development and understanding and significantly increase computational efficiency by elucidating possible accuracy measures. As the bounds can be computed for a given dataset, provided feasible-set sampling is available, the accuracy bounds can be used to assess the quality of datasets to train approximate inverse maps.

\subsection{Outline} The remainder of this paper is as follows. In Section \ref{sec:main} we introduce the computational framework for inverse problems and corresponding forward problems. In Section \ref{sec:alg} we introduce the algorithms for computing the accuracy bounds in Section \ref{sec:mainthm}. In Sections \ref{sec:micro} and \ref{sec:satsr} we describe the forward problems for localization microscopy and super-resolution of multi-spectral satellite data and present numerical results validating our theoretical results. In Section \ref{sec:additionalassconv} we provide convergence results. Additional information and proofs of main results in this paper are provided in the supplementary material.  The library accompanying the computational framework is available at the following \href{https://github.com/nm19000/AccuracyBounds}{GitHub repository}.

\section{Main results}\label{sec:main}
In this section, we define a rigorous framework for inverse problems and present the corresponding main theoretical results. Firstly, Section \ref{sec:preliminaries} formulates inverse problems, secondly, Section \ref{sec:alg} introduces an algorithm for computing the lower bound and thirdly, Section \ref{sec:mainthm} establishes inequalities using the lower bound provided by the introduced algorithm. The proofs of the main results are referred to Section \textbf{SM} $1$.

\subsection{Problem formulation}\label{sec:preliminaries}
We define an inverse problem as a forward problem $(F,\mathcal{M}_1, \mathcal{E})$ equipped with the task from \eqref{eq:sampling10}. The sets $\mathcal{M}_1 \subset \mathcal{X}$ and $\mathcal{E} \subset \mathcal{Z}$ in \eqref{eq:sampling10} are non-empty and bounded in the (pseudo) metric spaces $(\mathcal{Z},d_{\mathcal{Z}})$ and $(\mathcal{X},d_{\mathcal{X}})$, where $d_{\mathcal{X}}$ is a pseudo-metric that satisfies the Heine-Borel property with respect to the set of equivalence classes $[x]:=\{z \in \mathcal{X}: d_{\mathcal{X}}(x,z)=0\}$ for $x \in \mathcal{X}$, such that every closed and bounded set is compact in the quotient set $\mathcal{X}/\sim$. We consider compactness with respect to equivalence classes, as $d_{\mathcal{X}}$ is a pseudo-metric and is not positive definite. $\pi_1\colon \mathcal{M}_1 \times \mathcal{E} \to \mathcal{M}_1$ denotes the projection $(x,e) \mapsto x$. The disjoint union of sets $A,B \subset \mathcal{X}$ is denoted by $A\cupdot B$. An approximate inverse map, that from noisy measurements in the set $\mathcal{M}_2$ computes an approximation or a set of approximations $\hat{x} \subset \mathcal{X}$ to the true solution $x \in \mathcal{M}_1$, is denoted by 
\begin{align}\label{eq:decodermethod2}
    \begin{split}
\phi: \mathcal{M}_2 \rightrightarrows \mathcal{X},  \quad y \mapsto \phi(y) = \hat{x}.        \end{split}
\end{align}
\noindent The reconstruction error \eqref{eq:decodermethod2} for approximately solving \eqref{eq:sampling10} is often quantified by the empirical expectation of the residual of the reconstruction and a solution defined by \eqref{eq:empiricalerror}. For set-valued inverse maps, we consider the reconstruction error in \eqref{eq:empiricalerror} with respect to the Hausdorff distance given by
\begin{align*}
d^H_{\mathcal{X}}(A,B) := \max\left\{\sup_{x \in A}\inf_{z \in B}d_{\mathcal{X}}(x,z), \sup_{z \in B}\inf_{x \in A}d_{\mathcal{X}}(x,z)\right\},
\end{align*}
\noindent for $A, B \subset \mathcal{X}$. For a singleton $x \in \mathcal{X}$ the Hausdorff distance reduces to $d^H_{\mathcal{X}}(A,x) = \sup_{z \in A}d_{\mathcal{X}}(x,z)$. In \eqref{eq:empiricalerror} explicitly only the samples $\{(x_m,y_m)\}_{m=1}^M \subset \mathcal{M}_1 \times \mathcal{M}_2$ are needed and not the noise samples $\{e_m\}_{m=1}^M \in \mathcal{E}$.\newline

\begin{minipage}[b]{0.41\textwidth}
\hspace{-10pt}
\begin{tikzpicture}[node distance=0.9cm,
    every node/.style={fill=white}, align=left,
	base/.style = {rectangle, draw=gray,
                           text centered, font=\small}  
    ]
      
      \node [base, rounded corners, xshift=-1.1cm] (sth) {Dataset $\mathcal{M}_1$};
      
      \node [base, right of=sth, xshift = 1cm, yshift=2.1cm] (fm) {Forward Model $F$};

      \node [base, rounded corners, left of=fm, xshift = -1cm, yshift=-1cm] (n) {Noise $\mathcal{E}$};
      
      \node [base, below of=fm,yshift=-3.1cm] (im) {Inverse Map $\phi$};
      
      \node [base, rounded corners, right of=fm, xshift = 0.7cm, yshift=-2.1cm] (sthel) {Dataset $\mathcal{M}_2$};

	\draw[thick, ->,color=black]  (sth) -- node {\textcolor{black}{$x$}}(fm);
    \draw[thick, ->,color=black]  (n) -- node {\textcolor{black}{$e$}}(fm);
	\draw[thick, ->,color=black]  (fm) -- node {\textcolor{black}{Easy: $y=F(x,e)$}} (sthel);
	\draw[thick, ->,color=black]  (sthel) -- node {\textcolor{black}{$y$}} (im);
	\draw[thick, ->,color=black]  (im) -- node {\textcolor{black}{Ill-posed: find $x$.}} (sth);
\end{tikzpicture}
\captionof{figure}{Inverse problem framework.}
\end{minipage}
\hspace{8pt}
\begin{minipage}[b]{0.505\textwidth}
The forward model $F: \mathcal{M}_1\times \mathcal{E}\rightarrow \mathcal{M}_2$ in \eqref{eq:sampling10} is a map onto its image, $\mathcal{M}_2 = F(\mathcal{M}_1\times \mathcal{E})$. Examples of the forward model include, but are not limited to,
\begin{align}\label{eq:noisemodels}
\begin{split}
	F(x,e) &= G(x) + e,    \\
	F(x,e) &= G(x)\odot e, \\
	F(x,e) &= G(x)\odot e_1 + e_2, 
\end{split}
\end{align}
where $G \colon \mathcal{X} \to \mathcal{Y}$ is a linear or non-linear forward model, and the ambient set $\mathcal{Z}$ of the set of noise $\mathcal{E}$ depends on the considered model in such a way that the point-wise multiplication $\odot$ is defined.
\end{minipage}

\subsection{Algorithms for computing the average kernel sizes}\label{sec:alg}

Now we introduce an algorithm for computing the lower accuracy bound - the average kernel size - to the reconstruction error \eqref{eq:empiricalerror} and the optimal accuracy established in Theorem \ref{thm:lowerboundapprx}. The lower bound in Theorem \ref{thm:lowerboundapprx} may be non-zero, as we assume that given a noisy measurement $y \in \mathcal{M}_2$, there can exist multiple feasible solutions $x \in \mathcal{M}_1$ that the forward model maps to the measurement with some realization of noise. As in \cite{gottschling2023existence} we denote the set of possible solutions $x\in\mathcal{M}_1$ corresponding to a given $y\in\mathcal{M}_2$ as the \emph{feasible set} with 
\begin{align}\label{eq:feasset}
	F_y:=\pi_1(F^{-1}(y)) = \{x \in \mathcal{M}_1: \exists e \in \mathcal{E} \text{ s.t. } F(x,e)=y\}.
\end{align}
For many inverse problems \eqref{eq:sampling10} the feasible sets are non-singletons and there can exist at least two signals $x,x' \in \mathcal{M}_1$ and noise realizations $e,e' \in \mathcal{E}$ that map to the same noisy measurement $y = F(x,e) = F(x',e')$. Methods for solving ill-posed inverse problems \eqref{eq:sampling10} rely on exact or approximate inverse maps $\phi: \mathcal{M}_2 \rightarrow \mathcal{X}$ that take noisy measurements $y \in \mathcal{M}_2$ as inputs and provide an approximation $\hat{x}$ to a solution $x\in F_y$. The core issue is that no exact or approximate inverse map $\phi: \mathcal{M}_2 \rightrightarrows \mathcal{X}$ can distinguish between elements in the feasible set $F_y$ of a given measurement $y \in \mathcal{M}_2$. The basic observation that the feasible set \eqref{eq:feasset} contains more than one element motivates the inequality in Theorem \ref{thm:lowerboundapprx}. Algorithm \ref{alg:feasset} approximates the feasible sets \eqref{eq:feasset} for a set of measurements, as depicted in Figure \ref{fig:allocfeas}. 
\begin{figure}[h!]
\centering
\begin{tikzpicture}[node distance=4.2cm,
    every node/.style={fill=white}, align=center,
	base/.style = {rectangle, rounded corners, draw=gray,
                           text centered, font=\small}    
    ]

      \node [base] (sy) {Samples $\{y_k\}_{k=1}^K \subset \mathcal{M}_2$};

    \node [base, minimum height = 3.1cm, minimum width = 6.6cm, right of=sy, xshift = 1.25cm, yshift=-0.3cm] (fs1x) { };
 
    \node [base, draw=white, right of=sy, xshift = 1.3cm, yshift=-1.4cm] (fs13x) {Samples $\left\{\{(x_{k,n},e_{k,n})\}_{n=1}^{N(k)}\right\}_{k=1}^K \subset \mathcal{M}_1\times\mathcal{E}$};
      
    \node [base, right of=sy, xshift = 2.9cm, yshift=0.7cm] (fs1) {Feasible set $F_{y_1}$};


    \node [base, right of=sy, xshift = 2.9cm, yshift=0cm] (fsi) {Feasible set $F_{y_k}$};
    \node [base, right of=sy, xshift = 2.9cm, yshift=-0.7cm] (fsI) {Feasible set $F_{y_K}$};

	\draw[thick, ->,color=black]  (fs1) -- node {\textcolor{black}{ $y_1=F(x_{1,n},e_{1,n})$}}(sy);
	\draw[thick, ->,color=black]  (fsi) -- node {\textcolor{black}{...}} (sy);
	\draw[thick, ->,color=black]  (fsI) -- node {\textcolor{black}{ $y_K=F(x_{K,n},e_{K,n})$}} (sy);
\end{tikzpicture}
\vspace{-2pt}
\caption{Allocation of $x$-components of samples into feasible sets with Algorithm \ref{alg:feasset}. Note that the noise components can enlarge the feasible sets for additive and multiplicative noise models. This allocation does not require any distributional assumptions and only relies on the problem formulation in Section \ref{sec:preliminaries}.}\label{fig:allocfeas}
\vspace{-5pt}
\end{figure}

\noindent Algorithm \ref{alg:feasset} requires a number $K \in \mathbb{N}$ of measurement samples in $\mathcal{M}_2$, a description of the dataset $\mathcal{M}_1\times \mathcal{E}$ and a non-injective forward model $F$ and returns a list of approximations to feasible sets $\left\{F_{y_k}^{N(k)}\right\}_{k=1}^K$, their sizes $\{N(k)\}_{k=1}^K$ and an extended dataset $\mathcal{D}_M$. In summary, the sampling notation denotes $K$ as the number of measurement instances $y_k$, $N(k)$ the feasible-set sample count for a measurement $y_k$, and $M=\sum_{k=1}^K N(k)$ the total number of pairs $(x_m,y_m)$ in the dataset $\mathcal{D}_M$. To compute the average kernel size, there are no restrictions on how the samples are acquired in Algorithm \ref{alg:feasset}, such as restricting to Monte Carlo or Markov chain Monte Carlo (MCMC) \cite{liu2001monte, casella2008monte}. Section \ref{sec:applications} demonstrates that this framework is applicable to complicated forward models, such as those in localization microscopy, and for very large data dimensions $d_1,d_2 \in \mathbb{N}$ where $\mathcal{X}=\mathbb{R}^{d_1}$ and $\mathcal{Y}=\mathbb{R}^{d_2}$, such as in multi-spectral satellite data. Given samples of measurements from an inverse problem \eqref{eq:sampling10}, the feasible sets for each measurement sample are approximated from below in Algorithm \ref{alg:feasset}. This step can be implemented in various ways, yet the samples of $x$-components all have to satisfy $F(x,e)= y$. This requires finding possible solutions and instances of noise that the forward model maps to the same measurement. In practice, the computational cost is therefore governed by how feasible-set samples can be obtained for the forward model at hand; our framework is effective when such samples are accessible - through a generative simulator, MCMC/posterior sampling, or the use of a linear-algebraic structure as in Theorem~\ref{lem:lemmasym}). The different possible samples of noise are a crucial part of any measurement and these are relevant in Algorithm \ref{alg:feasset}, as they increase the possible samples of $x$-components that can map to this measurement. \newline
\vspace{-5pt}
\begin{algorithm}[h!]
\caption{Algorithm for approximating the feasible sets}\label{alg:feasset}
\begin{algorithmic}[0]
\Require  $K, N_{\mathrm{max}} \in \mathbb{N}$, $\mathcal{M}_2$, $\mathcal{M}_1\times \mathcal{E}$, $F$
\State $\mathcal{D}_M \gets \emptyset$, $M \gets 0$
\State Initialize lists $\mathbf{N}\gets \{\,\}$ and $\mathbf{F}\gets\{\,\}$
\For{$k \in \{1,\dots,K\}$}
    \State $y_k \in \mathcal{M}_2$, 
    \Comment{The samples $y_k$ can be inputs or sampled during the algorithm.}
    \State $F_{y_k}^{N(k)} \gets \emptyset$
    \State $N(k) \gets 0$
    \While{$N(k) < N_{\mathrm{max}}$}
     \Comment{Threshold for feasible set samples.}
    \State{$(x_{k,n},e_{k,n}) \in \mathcal{M}_1\times \mathcal{E}$}
    \Comment{Depending on problem acquire samples.}
    \If{$F(x_{k,n},e_{k,n})= y_k$}
    \Comment{Implementation is forward model dependent.}
    \State $F_{y_k}^{N(k)} \gets F_{y_k}^{N(k)} \bigcup \{x_{k,n}\}$
    \State $\mathcal{D}_M \gets \mathcal{D}_M \cupdot \{(x_{k,n}, y_k)\}$
    \Comment{This relates $(x_{k,n}, y_k)$ to $(x_m, y_m)$ in \eqref{eq:empiricalerror}.}
    \State $N(k) \gets N(k)+1$
    \EndIf 
    \EndWhile
    \State $\mathbf{N}[k] \gets N(k)$
    \State $\mathbf{F}[k] \gets F_{y_k}^{N(k)}$
    \State $M \gets M+N(k)$
\EndFor 
\Return  $\mathbf{F}=\left\{F_{y_k}^{N(k)}\right\}_{k=1}^K$, $\mathbf{N}=\{N(k)\}_{k=1}^K$, $\mathcal{D}_M = \{(x_m,y_m)\}_{m=1}^M$
\end{algorithmic}
\end{algorithm}

\begin{remark}[Distribution-free bounds vs. sampling-dependent estimates]
In the while loop of Algorithm \ref{alg:feasset}, samples $(x_{k,n},e_{k,n}) \in \mathcal{M}_1\times \mathcal{E}$ are obtained. There is no restriction on how these samples of $x$- and noise components are acquired. For example, consider the set of bounded noise in a closed ball around zero and note that this set can be equipped with different distributions to model the noise. The closed ball around zero contains exactly the same instantiations of noise vectors, albeit with different frequencies. As Theorem \ref{thm:lowerboundapprx} is an analytical result, independent of any probabilistic setting, it is valid in both cases. If the sampling procedure is biased toward certain noise realizations (e.g. drawn from a specific noise distribution), then the numerical estimate may reflect this bias. In applications, the feasible-set sampling strategy can affect the numerical value, but not the validity of the inequality in Theorem \ref{thm:lowerboundapprx}.
\end{remark}

\begin{definition}[Average Kernel Size] 
Let $p \in [1,\infty)$. Let $(F, \mathcal{M}_1, \mathcal{E})$ be a forward problem. Fix $K, N_{\mathrm{max}} \in \mathbb{N}$ and apply Algorithm \ref{alg:feasset} to the forward problem $(F, \mathcal{M}_1, \mathcal{E})$ to obtain $\left\{\{x_{k,n}\}_{n=1}^{N(k)}\right\}_{k=1}^K = \left\{F_{y_k}^{N(k)}\right\}_{k=1}^K$, $\{N(k)\}_{k=1}^K$, $\mathcal{D}_M$. Let $N(k) \ge 2$ for all $k=1, \ldots K$. The average kernel size is defined as 
\begin{align}\label{eq:lowerboundogboiz0}
   \operatorname{Kersize}(F,\mathcal{M}_1,\mathcal{E},p)_K= \left(\tfrac{1}{M}\sum_{k=1}^K \tfrac{2}{N(k)-1} \sum_{1 \leq n < n' \leq N(k)}d_{\mathcal{X}}( x_{k,n},x_{k,n'})^p\right)^{1/p},
\end{align} 
where $M = \sum_{k=1}^KN(k)$.
\end{definition}

\noindent The inner double sum in \eqref{eq:lowerboundogboiz0} does not include the $N(k)$ diagonal terms that are zero and, hence, we divide by a factor $M(N(k)-1)$. If $N(k)=N(K)$ for all $k\in\{1,\dots,K\}$, this is a normalization by a factor $KN(K)(N(K)-1)$. Then, \eqref{eq:lowerboundogboiz0} reduces to
\begin{align}\label{eq:lowerboundogboiz}
   \operatorname{Kersize}(F,\mathcal{M}_1,\mathcal{E},p)_K^p= \tfrac{1}{K}\sum_{k=1}^K \tfrac{2}{N(K)(N(K)-1)} \sum_{1 \leq n < n' \leq N(K)} d_{\mathcal{X}}(x_{k,n},x_{k,n'})^p.
\end{align} 
The quantity in \eqref{eq:lowerboundogboiz0} is closely related to \cite[Definition 3.8]{gottschling2023existence}, but does not require distributional or measurability assumptions. The average kernel size measures the average pairwise distance between signal components of pairs $(x,e),(x',e')\in\mathcal{M}_1\times\mathcal{E}$ that generate the same measurement, $F(x,e)=F(x',e')=y_k$. For linear inverse problems, this corresponds to the average size of the null space intersected with the dataset difference set $\mathcal M_1-\mathcal M_1=\{x-x':x,x'\in\mathcal M_1\}$. This motivates the terminology "average kernel size" for general nonlinear forward models. 

\subsection{Sharp accuracy bounds for computing solutions to inverse problems}\label{sec:mainthm}

Our main theorem provides a lower bound to the reconstruction error \eqref{eq:empiricalerror} that any approximate inverse map for any inverse problem of the form \eqref{eq:sampling10} can attain. The lower bound, half the average kernel size, depends only on the feasible set samples and on the parameter $p \in [1,\infty)$, that is fixed by the relevant loss function, as well as evaluation tools, such as the pseudo-metric $d_{\mathcal{X}}$ on $\mathcal{X}$ and is independent of the particular approximate inverse map $\phi:\mathcal M_2\rightrightarrows \mathcal{X}$.  

\begin{theorem}[Accuracy Bounds for Computing Solutions to Inverse Problems]\label{thm:lowerboundapprx}
     Let $p \in [1,\infty)$. Let $(F, \mathcal{M}_1, \mathcal{E})$ be a forward problem. Fix $K, N_{\mathrm{max}} \in \mathbb{N}$ and apply Algorithm \ref{alg:feasset} to the forward problem $(F, \mathcal{M}_1, \mathcal{E})$ to obtain $\left\{F_{y_k}^{N(k)}\right\}_{k=1}^K$, $\{N(k)\}_{k=1}^K$ with $N(k) \geq 2$ for all $k \in \{1, \ldots, K\}$, $\mathcal{D}_M= \{(x_m,y_m)\}_{m=1}^M$ with $M = \sum_{k=1}^K N(k)$ and compute the average kernel size \eqref{eq:lowerboundogboiz0}. Then, \textbf{half the average kernel size is a lower bound} to the reconstruction error of any approximate inverse map $\phi: \mathcal{M}_2 \rightrightarrows \mathcal{X}$ on the dataset $\mathcal{D}_M$,
\begin{align*}
			\tfrac{1}{2}\operatorname{Kersize}(F,\mathcal{M}_1,\mathcal{E},p)_K
			 \leq  \mathcal{L}(\mathcal{D}_M, \phi,p).
\end{align*}
Additionally, we have that
\begin{align*}
			\tfrac{1}{2}\operatorname{Kersize}(F,\mathcal{M}_1,\mathcal{E},p)_K
			 \leq  \inf_{\phi} \mathcal{L}(\mathcal{D}_M, \phi,p) \leq \operatorname{Kersize}(F,\mathcal{M}_1,\mathcal{E},p)_K,
\end{align*}
\noindent where the infimum is taken over all approximate inverse maps $\phi: \mathcal{M}_2 \rightrightarrows \mathcal{X}$. One map that attains the infimum is an extension to $\mathcal{M}_2$ of a map satisfying
\begin{align*}
\theta(y_k) = \argmin_{z \in \mathcal{X}}  \sum_{n=1}^{N(k)}d_{\mathcal{X}}(x_{k,n},z)^p.
\end{align*}
\end{theorem}

\noindent Theorem \ref{thm:lowerboundapprx} in combination with Algorithm \ref{alg:feasset} provides a practical recipe for quantifying the lower accuracy bound to the reconstruction error of any approximate inverse map. This is useful in various settings. For example, for deep learning engineers that aim to fit an optimal model or network architecture as an approximate inverse map for a given inverse problem on a given dataset with supervised training. The bound in Theorem~\ref{thm:lowerboundapprx} is stated for the infimum over all, possibly unstable, inverse maps. Since Lipschitz continuous maps form a subclass, the same lower bound applies to Lipschitz decoders:
\begin{align*}
\tfrac{1}{2}\,\operatorname{Kersize}(F,\mathcal{M}_1,\mathcal{E},p)_K
\le \inf_{\phi}\mathcal{L}(\mathcal{D}_M,\phi,p)
\le \inf_{\substack{\phi:\mathcal{M}_2\to\mathcal X\\ L(\phi)<\infty}}\mathcal{L}(\mathcal{D}_M,\phi,p).
\end{align*}
The lower bound also holds for maps obtained by regularization as in \eqref{eq:regsolvers},
\begin{align*}
			\tfrac{1}{2}\operatorname{Kersize}(F,\mathcal{M}_1,\mathcal{E},p)_K
			 \leq  \mathcal{L}(\mathcal{D}_M, \phi_{\alpha},p) \quad \text{ for } \alpha \in (0,\infty).
\end{align*}
\noindent In Section \textbf{SM} $2$ we show that the lower bound is sharp - in the sense that for a specific forward model, dataset, and set of noise, as well as approximate inverse map, the lower bound is attained. Theorem \ref{thm:lowerboundapprx} also provides an upper bound, the average kernel size, on the accuracy of the reconstruction for the optimal approximate inverse maps that attain the infimum in the accuracy limit \eqref{eq:empiricalerrorlimit}. 

\begin{remark}[Application to Finite Datasets $\mathcal{M}_1$]
Consider $M' \in \mathbb{N}$ and $\mathcal{D}_{M'}= \{(x_m,y_m)\}_{m=1}^{M'} \subset \mathcal{M}_1\times \mathcal{M}_2$ with a given forward model $F$ and noise $\mathcal{E}$. Theorem \ref{thm:lowerboundapprx} in combination with Algorithm \ref{alg:feasset} now enable to compute accuracy bounds before fitting a model to a given dataset. The relation between the dataset $\mathcal{D}_{M'}=\{(x_m,y_m)\}_{m=1}^{M'} $ of size $M' \in \mathbb{N}$ and the allocation to feasible sets $\{F_{y_k}^{N(k)}\}_{k=1}^K$ and dataset $\mathcal{D}_M$ is provided by \ref{alg:feasset}, when $\mathcal{D}_M$ is initialized with $\mathcal{D}_{M'}$ instead of the empty set. By setting $I = |\pi_2(\mathcal{D}_{M'})|$, $\mathcal{M}_2 = \pi_2(\mathcal{D}_{M'})$ and $\mathcal{M}_1 = \pi_1(\mathcal{D}_{M'})$ and letting $N_{\mathrm{max}} = |\mathcal{M}_1| \in \mathbb{N}$ the full dataset can be allocated into feasible sets and a larger dataset of size $M \geq M'$. This means that the output dataset $\mathcal{D}_M$ of Algorithm \ref{alg:feasset} is a superset of the input data, $\mathcal{D}_{M'}\subset \mathcal{D}_M$. 
\end{remark}

\subsection{Fast lower bound}

Now let $\mathcal{X}=\mathbb{C}^{d_1}$ and $\mathcal{Y}=\mathbb{C}^{d_2}$ with dimensions $d_1,d_2 \in \mathbb{N}$ and equip $\mathbb{C}^{d_1}$ with a pseudo-norm $\|\cdot\|$. For forward models $F: \mathcal{M}_1\times \mathcal{E} \rightarrow \mathcal{M}_2$ of the form $F(x,e) = Ax+e$, with $A \in \mathbb{C}^{d_2\times d_1}$ and $(x,e) \in \mathcal{M}_1\times \mathcal{E}$, for a dataset of size $M$, Theorem \ref{thm:lowerboundapprx} can be optimized with an additional symmetry assumption from the computational complexity of $\mathcal{O}(M^2)$ down to $\mathcal{O}(M)$. In inner product spaces the symmetry assumption in \eqref{eq:reflected_null} is equivalent to assuming that the reflection of a vector around the orthogonal component to the kernel of the forward model is a realistic data point. If the forward model $F$ is a linear map on $\mathcal{M}_1\times \mathcal{E}$, we can compute the projection onto the kernel of $F$ with its Moore-Penrose inverse, \cite{korolev2020inverse,ben2006generalized}. In the following $P_{\mathcal{N}(F)}: \mathcal{M}_1 \times \mathcal{E} \rightarrow \mathcal{N}(F) \subset \mathcal{M}_1 \times \mathcal{E}$ denotes the projection onto the kernel of the forward model $F$ and is computed by $P_{\mathcal{N}(F)}=\mathds{1}-F^{\dagger}F$, where $F^\dagger$ is the Moore-Penrose inverse of $F$.

\begin{theorem}[Fast Lower Bound for Functions on the Measurement Data]\label{lem:lemmasym}
    Let $p \in [1,\infty)$. Let $(F, \mathcal{M}_1, \mathcal{E})$ be a forward problem with $F(x,e) = Ax+e$ with $A \in \mathbb{C}^{d_2\times d_1}$ for $(x,e) \in \mathcal{M}_1\times \mathcal{E}$. Let $\phi: \mathcal{M}_2 \rightarrow \mathbb{C}^{d_1}$ be an approximate inverse map. Let $\mathcal{D}_{M'}=\{(x_m,y_m)\}_{m=1}^{M'} \subset \mathcal{M}_1\times \mathcal{M}_2$. Extend $\mathcal{D}_{M'}$ to $\mathcal{D}_M$ by adding 
    \begin{align}\label{eq:reflected_null}
    (x',y_m) = (\pi_1\left((x_m,e_m)-2P_{\mathcal{N}(F)}(x_m,e_m)\right), y_m) \text{ to } \mathcal{D}_M,
    \end{align}
    for all $x_m \in \pi_1(\mathcal{D}_{M'})$ with $e_m \in \mathcal{E}$ such that $F(x_m,e_m)=y_m$ and $x' \in \mathcal{M}_1$.
    Then,
\begin{itemize}    
   \item[(i)] we have that the \textbf{average symmetric kernel size} is a lower bound to the reconstruction error of any approximate inverse map $\phi$ on the dataset $\mathcal{D}_M= \{(x_m,y_m)\}_{m=1}^M$, with $M=2M'$
\begin{align*}
			\operatorname{SKersize}(F,\mathcal{M}_1,\mathcal{E},p)_M:= \left(\frac{1}{M'} \sum_{m=1}^{M'} \|v_{m}\|^p\right)^{1/p}
			 \leq \mathcal{L}(\mathcal{D}_M, \phi,p),
\end{align*}
 with $\|v_{m}\| = \|\pi_1(P_{\mathcal{N}(F)}(x_{m},e_{m}))\|$
for $m \in \{1, \dots,M'\}$.
\item[(ii)] additionally, we have that
\begin{align*}
		\operatorname{SKersize}(F,\mathcal{M}_1,\mathcal{E},p)_M
			 \leq \inf_{\phi: \mathcal{M}_2 \rightarrow \mathbb{C}^{d_1}} \mathcal{L}(\mathcal{D}_M, \phi,p) \leq 2  	\operatorname{SKersize}(F,\mathcal{M}_1,\mathcal{E},p)_M,
\end{align*}
\noindent where the infimum is taken over all approximate inverse maps $\phi: \mathcal{M}_2 \rightarrow \mathbb{C}^{d_1}$.
\end{itemize}
\end{theorem}

\section{Computing accuracy bounds in applied inverse problems}\label{sec:applications}

We apply the framework from Section \ref{sec:preliminaries} to the inverse problems that arise in localization microscopy and super-resolution of multi-spectral satellite data. We formulate the inverse problems in the framework and empirically verify Theorem \ref{thm:lowerboundapprx} and Theorem \ref{lem:lemmasym} by computing the lower and upper accuracy bounds based on the average (symmetric) kernel size and, respectively, the loss.

\subsection{Accuracy bounds in microscopy}\label{sec:micro}
Localization microscopy aims to localize point-like light-sources at the highest possible resolution that attains the system's inherent accuracy limits \cite{lee2017unraveling,sgouralis2024bnp}. With the average kernel size, the upper and lower bounds of the optimum localization error in microscopy can be quantified. In previous work in localization microscopy, the inverse of the Fisher information matrix is computed as a lower bound to the variance of an estimator for the localization \cite{ober2004localization}. The work of \cite{ober2004localization} uses the classical Cramer-Rao bound \cite{cramer1999mathematical,zacks1971theory,rao1973linear,kay1993fundamentals}, which bounds the variance of an unbiased estimator by the inverse Fisher information matrix. For the derivation of this lower bound in \cite{ober2004localization} the following assumptions are made. Firstly, the estimator is unbiased (i.e., an estimation procedure whose mean produces the correct result), second, the spatial coordinates of the detected photons are independent and identically distributed random variables, and thirdly, the counting process and the random variables that describe the spatial coordinates are mutually independent. Unfortunately, these assumptions are not satisfied in many applications of interest to practitioners \cite{lee2017unraveling}. For example, as there may exist multiple possible locations corresponding to a measurement data sample, obtaining a deterministic unbiased estimator may not be possible, even with a very large sample size. As Theorem \ref{thm:lowerboundapprx} holds under more general assumptions, it provides a lower bound to the localization accuracy of any estimator in this context. The best estimators for a given dataset can be determined using the lower and upper bounds of Theorem \ref{thm:lowerboundapprx}.

\subsubsection{Localization microscopy as an inverse problem}

We give a brief conceptual introduction to single-molecule localization microscopy as an inverse problem. For further details, see the reviews on \cite{lee2017unraveling,von2017three}. In localization microscopy, one objective is to recover the position of a particle in a volume $\mathcal{V} \subset \mathbb{R}^3$ and the background photon flux $C \in [0,C_{max}]$ and the particle photon emission rate $h \in [0, h_{max}]$ given a detection discretized into $d^x_2\times d^y_2$ pixels for $d_2^x, d^y_2 \in \mathbb{N}$. In summary, one tries to reconstruct $\Theta = (x,y,z, C, h)$ from noisy measurements of pixel-wise intensities at pixels $(p_x,p_y) \in \mathcal{I}_{d_2} = \{1, \dots,d^x_2\} \times \{1, \dots,d_2^y\}$. Hence, the set of signals is given by $\mathcal{M}_1 = \mathcal{V} \times [0,C_{max}] \times [0,h_{max}] \subset \mathcal{X}=\mathbb{R}^5$, the set of noise is given by $\mathcal{E} = \mathcal{B}(0,\mathds{1}\epsilon_1) \times \mathcal{B}(0,\mathds{1}\epsilon_2) \subset \mathbb{R}^{2(d^x_2\times d^y_2)}$ and the forward model is given by
\begin{align}\label{eq:forwardmicro}
\begin{split}
    F : \mathcal{M}_1 \times \mathcal{E} \rightarrow [0,I_{max}]^{d^x_2\times d^y_2}, 
    (\Theta, e) \mapsto F(\Theta, e) :=  \left(w_{(p_x,p_y)}(\Theta, e)\right)_{(p_x,p_y) \in \mathcal{I}_{d_2}},
    \end{split}
\end{align}
where $I_{max} \in (0,\infty)$ is the maximum measured intensity. The measured intensity $w_{(p_x,p_y)}(\Theta, e)$ for $(p_x,p_y)\in\mathcal{I}_{d_2}$ results from a mixture of additive and multiplicative noise and a non-linear forward model, for details, see Section \textbf{SM} $3.1$. In this setting, the inverse problem in the framework of \eqref{eq:sampling10} is given by
\begin{align}\label{eq:samplingmicro}
    \text{recover } \Theta\text{ given data } F(\Theta, e) \text{ of } \Theta \in \mathcal{M}_1 \subset \mathbb{R}^5 \text{ and }  e \in \mathcal{E}\subset \mathbb{R}^{2(d^x_2\times d^y_2)}.
\end{align}
\noindent The key point here is that the inverse problem of localization microscopy in \eqref{eq:samplingmicro} satisfies our central assumption \eqref{eq:centralassumption}. Two different $\Theta, \Theta' \in \mathcal{M}_1$ with $\Theta \neq \Theta'$ and noise samples $e,e' \in \mathcal{E}$ can be mapped to the same noisy measurement 
\begin{align}\label{eq:illposedmicro}
F(\Theta, e) = F(\Theta', e').
\end{align}
\noindent \eqref{eq:illposedmicro} states that two correct results can lead to exactly the same noisy measurement. As there is more than one possible correct result, and, the average kernel size is non-zero, and, hence, there is no estimation procedure that can yield a mean which is the unique correct result. As \eqref{eq:centralassumption} is satisfied, this motivates computing the accuracy bounds, the average kernel sizes, and validating the lower bound in Theorem \ref{thm:lowerboundapprx}.

\subsubsection{Numerical experiments for computing the average kernel size}

\begin{table}[h!]
\resizebox{\linewidth}{!}{%
\begin{tabular}{l|cc|cc|cc|cc}
\hline
Estimators & A1 & & A2 & & A3 & & A4 \\
\hline
Dataset Size & $25$ & $25\times 1501$ & $25$ & $25 \times 1501$ & $25$ & $25 \times 1501$  & $25$ & $25 \times 1501$ \\
\hline
Mean & $2.917$ & $2.634$ & $3.905$ & $3.965$ & $6.559$ & $6.296$ & $12.195$ & $11.635$  \\
Median & $2.634$ & $2.634$ & $3.965$ & $3.965$ & $6.294$ & $6.294$ & $ 11.632$ & $11.632$\\
MAP & $3.611$ & $3.611$ & $5.089$ & $5.089$ & $9.318$ & $9.318$ & $16.923$ & $ 16.923$ \\
ML  & $4.248$ & $3.719$ & $5.338$ & $5.812$ & $9.453$ & $8.153$ & $17.009$ & $16.675$ \\
 \hline
\hline
$\frac{1}{2}\operatorname{Kersize}$  & $\qquad$ & $1.864$ & $\qquad$ & $2.805$ & $\qquad$ & $4.452$ & $\qquad$ & $8.228$  \\
\hline
\end{tabular}
}
\caption{RMSE in $[nm]$ for different estimators on the dataset consisting of the ground truth and on the dataset output of Algorithm \ref{alg:feasset} for microscope set-ups $\{A_1,A_2,A_3,A_4\}$. The last row is the corresponding lower accuracy bound, half the average kernel size $\frac{1}{2}\operatorname{Kersize}(F,\mathcal{M}_1,\mathcal{E},2)_K$, for microscope set-ups $\{A_1,A_2,A_3,A_4\}$ with $K=25$ and $N(K)=1501$.}\label{tab:lossfscm}
\vspace{-5pt}
\end{table}

\begin{figure}[h!]
\begin{subfigure}{\textwidth}
\centering
\includegraphics[width=0.9\textwidth]{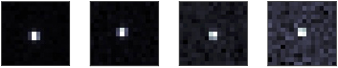}
\caption{Different noisy measurements per microscope set-up $\{A_1,A_2,A_3,A_4\}$ with increasing  background photon flux and decreasing particle photon emission rate from left to right.} \label{fig:estimatorvskerfscm_pics}
~\\\vspace{-10pt}
\end{subfigure}
\begin{subfigure}{0.50\textwidth}
\begin{tikzpicture}[scale=0.75]
	\begin{axis}[legend pos=south east,%
    xmax=10, xmin=1, ymax=25, ymin=1, 
    title={\large \textbf{Mean Estimator}},
  xlabel={$\frac{1}{2}\operatorname{Kersize}(F,\mathcal{M}_1,\mathcal{E},2)_1$ [nm]},
  ylabel={RMSE of Mean Estimator [nm]},
	scatter/classes={%
		A2={mark=o,blue},%
		A3={mark=o,cyan},
        A4={mark=o,Aquamarine},
        A5={mark=o,NavyBlue}}%
        ]
	\addplot[scatter,only marks,%
		scatter src=explicit symbolic]%
	table[meta=label] {
x    y label
  1.410391  2.012727    A2
  1.262702  1.795209    A2
   1.470487  2.089624    A2
  1.440023  2.040633    A2
   1.241982  1.756619    A2
  1.265349  1.789438    A2
  1.318187  1.881723    A2
  1.434148  2.034203    A2
  1.400375 1.981509    A2
   1.362142  1.927547    A2
 1.350037  1.908635    A2
  1.429750  2.029405    A2
 1.451891 2.076455    A2
  1.267500  1.804334    A2
 1.228678  1.740209    A2
  1.381924 1.965922    A2
 1.482610 2.101450    A2
 1.254132 1.779731    A2
 1.170161 1.656826    A2
  1.425620  2.031788    A2
 1.334797  1.891618    A2
 1.244329  1.759771    A2
  1.167626  1.660904    A2
  1.553520  2.199503    A2
1.181366  1.673204    A2
   2.729674  3.859451    A3
   3.030944  4.285021    A3
  2.716648 3.840719    A3
   2.849334  4.028307    A3
   2.822844  3.991414    A3
   2.821014  3.988723    A3
   2.848923  4.030836    A3
   2.747234  3.883970    A3
   2.618895  3.702694    A3
   2.918310  4.125995    A3
  2.880222  4.071986    A3
  2.859253  4.042698    A3
  2.937669  4.153428    A3
  2.637573  3.730195    A3
  2.632219  3.721469    A3
  2.958203  4.182644    A3
  2.885864  4.081676    A3
  2.678167  3.786453    A3
  3.020559  4.270644    A3
  2.722384  3.848772    A3
  2.670885  3.777233    A3
  2.766845  3.913119    A3
  2.742253  3.878445    A3
  2.804215  3.964532    A3
  2.751720  3.891133    A3
   4.481023  6.335915    A4
   4.587538  6.488511    A4
   4.109176  5.813312    A4
   4.351643  6.152180    A4
   4.030238  5.698186    A4
   4.602496  6.506829    A4
   4.490561  6.349157    A4
   4.715957  6.670419    A4
   4.113106  5.816764    A4
  4.578288  6.472727    A4
  4.719583  6.675730    A4
  4.261239  6.028870    A4
  4.546144  6.427689    A4
  4.457560  6.301856    A4
  4.462218  6.310699    A4
  4.220948  5.972678    A4
  4.628749  6.545125    A4
  4.321630  6.110246    A4
  4.679321  6.617145    A4
  4.224017  5.972094    A4
  4.669143  6.605586    A4
  4.378358  6.190033    A4
  4.310618  6.094651    A4
  4.254440  6.015384    A4
  4.969526  7.026013    A4
   8.059731  11.402658    A5
   8.409681  11.892112    A5
   9.037278  12.780899    A5
   7.601414  10.746629    A5
   8.978950  12.694528    A5
   8.342168  11.795354    A5
   8.111467  11.467711    A5
   9.167278  12.960414    A5
   8.756119  12.379135    A5
   7.829482  11.076340    A5
  8.168964  11.555059    A5
  7.491770  10.592497    A5
  7.651813  10.828108    A5
  8.336915  11.791145    A5
  8.090838  11.438895    A5
  8.441580  11.938763    A5
  7.756363  10.965897    A5
  7.930849  11.214834    A5
  8.588233  12.147812    A5
  7.136297  10.091163    A5
  8.072102  11.415063    A5
  7.948745  11.238478    A5
  8.380850  11.848393    A5
  8.551694  12.090586    A5
  8.500126  12.017741    A5
	};
    \addplot[lightgray, thick, domain=0:12] (x,x);
      \addplot[lightgray, thick, domain=0:12] (x,2*x);
    \legend{A1, A2, A3, A4}
	\end{axis}
\end{tikzpicture}
\end{subfigure}\hfill
\begin{subfigure}{0.48\textwidth}
\begin{tikzpicture}[scale=0.75]
	\begin{axis}[legend pos=south east,%
    xmax=10, xmin=1, ymax=25, ymin=1, 
    title={\large \textbf{Median Estimator}},
  xlabel={$\frac{1}{2}\operatorname{Kersize}(F,\mathcal{M}_1,\mathcal{E},2)_1$ [nm]},
	scatter/classes={%
		A2={mark=o,blue},%
		A3={mark=o,cyan},
        A4={mark=o,Aquamarine},
        A5={mark=o,NavyBlue}}%
        ]
	\addplot[scatter,only marks,%
		scatter src=explicit symbolic]%
	table[meta=label] {
x    y label
   8.059731  11.394387    A5
   8.409681  11.889124    A5
   9.037278  12.776386    A5
   7.601414  10.746440    A5
   8.978950  12.693925    A5
   8.342168  11.793681    A5
   8.111467  11.467528    A5
   9.167278  12.960171    A5
   8.756119  12.378898    A5
   7.829482  11.068873    A5
  8.168964  11.548821    A5
  7.491770  10.591437    A5
  7.651813  10.817704    A5
  8.336915  11.786251    A5
  8.090838  11.438360    A5
  8.441580  11.934221    A5
  7.756363  10.965501    A5
  7.930849  11.212180    A5
  8.588233  12.141553    A5
  7.136297  10.088887    A5
  8.072102  11.411873    A5
  7.948745  11.237486    A5
  8.380850  11.848364    A5
  8.551694  12.089895    A5
  8.500126  12.016990    A5
4.481023	6.335012	A4
4.587538	6.485599	A4
4.109176	5.809317	A4
4.351643	6.152102	A4
4.030238	5.697718	A4
4.602496	6.506744	A4
4.490561	6.348496	A4
4.715957	6.667149	A4
4.113106	5.814874	A4
4.578288	6.472524	A4
4.719583	6.672280	A4
4.261239	6.024294	A4
4.546144	6.427075	A4
4.457560	6.301842	A4
4.462218	6.308427	A4
4.220948	5.967334	A4
4.628749	6.543860	A4
4.321630	6.109671	A4
4.679321	6.615361	A4
4.224017	5.971675	A4
4.669143	6.600968	A4
4.378358	6.189874	A4
4.310618	6.094105	A4
4.254440	6.014684	A4
4.969526	7.025631	A4
2.729674	3.859057	A3
3.030944	4.284975	A3
2.716648	3.840642	A3
2.849334	4.028225	A3
2.822844	3.990775	A3
2.821014	3.988188	A3
2.848923	4.027644	A3
2.747234	3.883883	A3
2.618895	3.702447	A3
2.918310	4.125739	A3
2.880222	4.071893	A3
2.859253	4.042247	A3
2.937669	4.153108	A3
2.637573	3.728851	A3
2.632219	3.721280	A3
2.958203	4.182137	A3
2.885864	4.079870	A3
2.678167	3.786240	A3
3.020559	4.270294	A3
2.722384	3.848750	A3
2.670885	3.775944	A3
2.766845	3.911606	A3
2.742253	3.876840	A3
2.804215	3.964438	A3
2.751720	3.890223	A3
1.410391	1.993995	A2
1.262702	1.785225	A2
1.470487	2.078975	A2
1.440023	2.035835	A2
1.241982	1.756159	A2
1.265349	1.788946	A2
1.318187	1.863778	A2
1.434148	2.027783	A2
1.400375	1.979813	A2
1.362142	1.925747	A2
1.350037	1.908802	A2
1.429750	2.021333	A2
1.451891	2.052759	A2
1.267500	1.792038	A2
1.228678	1.737323	A2
1.381924	1.953854	A2
1.482610	2.096114	A2
1.254132	1.773022	A2
1.170161	1.654544	A2
1.425620	2.015566	A2
1.334797	1.887063	A2
1.244329	1.759429	A2
1.167626	1.650975	A2
1.553520	2.196277	A2
1.181366	1.670306	A2
	};
    \addplot[lightgray, thick, domain=0:12] (x,x);
      \addplot[lightgray, thick, domain=0:12] (x,2*x);
    \legend{A1, A2, A3, A4}
	\end{axis}
\end{tikzpicture}
\end{subfigure}
\begin{subfigure}{0.50\textwidth}
\begin{tikzpicture}[scale=0.75]
	\begin{axis}[legend pos=south east,%
    xmax=10, xmin=1, ymax=25, ymin=1, 
    title={\large \textbf{MAP Estimator}},
  xlabel={$\frac{1}{2}\operatorname{Kersize}(F,\mathcal{M}_1,\mathcal{E},2)_1$ [nm]},
  ylabel={RMSE of MAP Estimator [nm]},
	scatter/classes={%
		A2={mark=o,blue},%
		A3={mark=o,cyan},
        A4={mark=o,Aquamarine},
        A5={mark=o,NavyBlue}}%
        ]
	\addplot[scatter,only marks,%
		scatter src=explicit symbolic]%
	table[meta=label] {
x    y label
1.410391	2.026694	A2
1.262702	2.392322	A2
1.470487	3.780477	A2
1.440023	2.080749	A2
1.241982	1.764034	A2
1.265349	2.636432	A2
1.318187	3.073295	A2
1.434148	2.101903	A2
1.400375	2.521189	A2
1.362142	3.198665	A2
1.350037	1.923810	A2
1.429750	2.300238	A2
1.451891	2.371267	A2
1.267500	2.160331	A2
1.228678	1.890917	A2
1.381924	3.334087	A2
1.482610	2.627081	A2
1.254132	1.844168	A2
1.170161	1.953265	A2
1.425620	2.983426	A2
1.334797	2.145176	A2
1.244329	2.073553	A2
1.167626	2.127454	A2
1.553520	2.342070	A2
1.181366	1.892506	A2
2.729674	4.987865	A3
3.030944	4.957930	A3
2.716648	4.551524	A3
2.849334	6.112926	A3
2.822844	6.795316	A3
2.821014	4.052720	A3
2.848923	4.618461	A3
2.747234	4.796604	A3
2.618895	3.792746	A3
2.918310	4.696338	A3
2.880222	6.124748	A3
2.859253	4.731518	A3
2.937669	5.866558	A3
2.637573	3.974557	A3
2.632219	3.893485	A3
2.958203	7.092171	A3
2.885864	4.321750	A3
2.678167	4.938794	A3
3.020559	4.677402	A3
2.722384	6.202092	A3
2.670885	4.722587	A3
2.766845	4.491862	A3
2.742253	4.924923	A3
2.804215	5.687282	A3
2.751720	4.309244	A3
4.481023	9.330971	A4
4.587538	12.776652	A4
4.109176	7.234584	A4
4.351643	9.319180	A4
4.030238	7.217898	A4
4.602496	7.385255	A4
4.490561	7.173124	A4
4.715957	9.156677	A4
4.113106	8.401472	A4
4.578288	7.626792	A4
4.719583	11.157311	A4
4.261239	6.767768	A4
4.546144	7.065387	A4
4.457560	10.609736	A4
4.462218	8.609679	A4
4.220948	10.361549	A4
4.628749	12.491363	A4
4.321630	10.182249	A4
4.679321	9.089064	A4
4.224017	6.344238	A4
4.669143	7.925203	A4
4.378358	11.815376	A4
4.310618	8.715852	A4
4.254440	13.029051	A4
4.969526	8.097671	A4
8.059731	24.747876	A5
8.409681	13.669057	A5
9.037278	14.172486	A5
7.601414	16.165371	A5
8.978950	14.990911	A5
8.342168	17.381171	A5
8.111467	12.478985	A5
9.167278	25.720593	A5
8.756119	13.317881	A5
7.829482	11.117803	A5
8.168964	12.595723	A5
7.491770	12.010074	A5
7.651813	12.608140	A5
8.336915	16.907628	A5
8.090838	16.655968	A5
8.441580	31.434419	A5
7.756363	13.375611	A5
7.930849	11.782283	A5
8.588233	14.443531	A5
7.136297	14.811504	A5
8.072102	17.414418	A5
7.948745	11.478945	A5
8.380850	18.963267	A5
8.551694	16.387448	A5
8.500126	20.761779	A5
	};
    \addplot[lightgray, thick, domain=0:12] (x,x);
      \addplot[lightgray, thick, domain=0:12] (x,2*x);
    \legend{A1, A2, A3, A4}
	\end{axis}
\end{tikzpicture}
\end{subfigure}\hfill~
\begin{subfigure}{0.48\textwidth}
\begin{tikzpicture}[scale=0.75]
	\begin{axis}[legend pos=south east,%
    xmax=10, xmin=1, ymax=25, ymin=1, 
    title={\large \textbf{ML Estimator}},
  xlabel={$\frac{1}{2}\operatorname{Kersize}(F,\mathcal{M}_1,\mathcal{E},2)_1$ [nm]},
	scatter/classes={%
		A2={mark=o,blue},%
		A3={mark=o,cyan},
        A4={mark=o,Aquamarine},
        A5={mark=o,NavyBlue}}%
        ]
	\addplot[scatter,only marks,%
		scatter src=explicit symbolic]%
	table[meta=label] {
x    y label
8.059731	15.737762	A5
8.409681	17.107240	A5
9.037278	22.057356	A5
7.601414	15.979494	A5
8.978950	13.153956	A5
8.342168	12.887188	A5
8.111467	22.074787	A5
9.167278	23.180081	A5
8.756119	13.037006	A5
7.829482	26.424774	A5
8.168964	14.849294	A5
7.491770	13.612145	A5
7.651813	14.395350	A5
8.336915	13.710389	A5
8.090838	20.069373	A5
8.441580	13.521588	A5
7.756363	15.767252	A5
7.930849	13.635500	A5
8.588233	15.074651	A5
7.136297	12.006530	A5
8.072102	15.304982	A5
7.948745	14.185506	A5
8.380850	12.865200	A5
8.551694	18.949329	A5
8.500126	16.875562	A5
4.481023	8.020135	A4
4.587538	7.704862	A4
4.109176	7.188476	A4
4.351643	7.690474	A4
4.030238	10.294590	A4
4.602496	7.727725	A4
4.490561	7.262242	A4
4.715957	7.675811	A4
4.113106	6.417126	A4
4.578288	11.360378	A4
4.719583	7.890722	A4
4.261239	6.554834	A4
4.546144	6.787838	A4
4.457560	6.688488	A4
4.462218	7.303681	A4
4.220948	7.338186	A4
4.628749	6.984971	A4
4.321630	6.396882	A4
4.679321	8.410396	A4
4.224017	7.791289	A4
4.669143	6.917923	A4
4.378358	6.315967	A4
4.310618	12.620811	A4
4.254440	9.528368	A4
4.969526	10.846077	A4
2.729674	7.559779	A3
3.030944	4.367353	A3
2.716648	5.429928	A3
2.849334	5.893497	A3
2.822844	7.202550	A3
2.821014	4.925721	A3
2.848923	4.921155	A3
2.747234	5.634320	A3
2.618895	3.882701	A3
2.918310	5.071361	A3
2.880222	9.290752	A3
2.859253	4.174634	A3
2.937669	7.853580	A3
2.637573	4.225218	A3
2.632219	3.851072	A3
2.958203	7.910931	A3
2.885864	5.222923	A3
2.678167	5.489237	A3
3.020559	5.431171	A3
2.722384	4.224715	A3
2.670885	7.094362	A3
2.766845	5.425106	A3
2.742253	6.282712	A3
2.804215    5.336000	A3
2.751720	4.233903	A3
1.410391	2.273644	A2
1.262702	2.271681	A2
1.470487	2.465922	A2
1.440023	2.304377	A2
1.241982	2.468190	A2
1.265349	3.090156	A2
1.318187	2.870934	A2
1.434148	2.088134	A2
1.400375	2.622067	A2
1.362142	2.192143	A2
1.350037	2.862813	A2
1.429750	3.028116	A2
1.451891	2.354567	A2
1.267500	3.371723	A2
1.228678	2.086104	A2
1.381924	3.177305	A2
1.482610	2.625466	A2
1.254132	2.488293	A2
1.170161	1.881632	A2
1.425620	3.046927	A2
1.334797	2.698440	A2
1.244329	2.282892	A2
1.167626	1.790571	A2
1.553520	3.272534	A2
1.181366	1.685829	A2
	};
    \addplot[lightgray, thick, domain=0:12] (x,x);
      \addplot[lightgray, thick, domain=0:12] (x,2*x);
    \legend{A1, A2, A3, A4}
	\end{axis}
\end{tikzpicture}
\end{subfigure}
\vspace{-3pt}
\caption{RMSE of the $(x,y)-$location from different estimators against half of the average kernel size computed with feasible sets of size $N(1)= 1501$ with $K=1$ for $25$ different noisy measurements per microscope set-up $\{A_1,A_2,A_3,A_4\}$. The lower and upper gray lines ($y=x$ and $y=2x$) visualize the lower and upper accuracy bounds. All RMSE values are above the computed lower accuracy bounds and validate Theorem \ref{thm:lowerboundapprx}, $(i)$. The mean and median estimator's RMSE are below the upper accuracy bound. The different microscope set-ups $\{A_1,A_2,A_3,A_4\}$ highlight the increase of the kernel size with worsening imaging settings - a higher background photon flux and decreasing particle photon emission rate.} \label{fig:estimatorvskerfscm_mic}
\vspace{-5pt}
\end{figure}

For fluorescence microscopy \eqref{eq:samplingmicro} any method providing an estimator of $\hat{\Theta}$ is an approximate inverse map $\hat{\Theta}: \mathcal{M}_2 \rightarrow \mathbb{R}^5$ taking noisy measurements $w \in \mathcal{M}_2$ as input and producing approximations to the location and background photon flux and particle photon emission rate $\Theta$, $\phi(w) = \hat{\Theta}$. We adapt the algorithm for allocating samples to the feasible sets Algorithm \ref{alg:feasset} for the particular forward model of  \eqref{eq:forwardmicro}. The condition of assigning samples to feasible sets $F(\Theta,e)= w(\Theta,e)$, is met by obtaining MCMC samples where the noise samples are within the noise model. Details on the MCMC sampling are given in Section \textbf{SM} $3.3$. We use samples from four imaging set-ups, where the different set-ups $\{A_1,A_2,A_3,A_4\}$ correspond to different background photon fluxes and particle photon emission rates. The best set-up is $A_1$, which has a low background photon flux and a high particle photon emission rate, and the worst is $A_4$. Using these samples, we compute the average kernel size using Algorithm \ref{alg:feasset} and the RMSE loss for different estimators. For each set-up $\{A_1,A_2,A_3,A_4\}$ we consider $K=25$ different noisy measurements $\{w_k\}_{k=1}^{25} \subset \mathcal{M}_2 = F(\mathcal{M}_1\times\mathcal{E})$. For each example measurement, we approximate the feasible sets with $N(25) = 1501$ MCMC samples. We use a pseudo metric induced by a pseudo-norm - the $\ell_2$ norm composed with the projection onto the first two spatial components $\pi_{1,2}(\Theta) = \Theta_{x,y} =(x,y)$. To validate the lower accuracy bound of Theorem \ref{thm:lowerboundapprx}, we compute the RMSE loss $\mathcal{L}(\mathcal{D}_M,\phi,2)$ from \eqref{eq:empiricalerror} for the mean, median, maximum a posteriori, and maximum likelihood estimators on the input datasets and the output datasets of Algorithm \ref{alg:feasset} in Table \ref{tab:lossfscm}. The scatter plots of the RMSE against half the average kernel size in Figure \ref{fig:estimatorvskerfscm_mic} validate Theorem \ref{thm:lowerboundapprx} for each measurement and microscope set-up. The mean and median estimators are consistently optimal and within the accuracy bounds of Theorem \ref{thm:lowerboundapprx}. A convergence study of the number of samples in the feasible set is provided in Figure \textbf{SM} $2$.

\subsection{Accuracy bounds for super-resolution of multi-spectral satellite data}\label{sec:satsr}
Single image super-resolution converts a low-resolution image with coarse details to a corresponding high-resolution image with higher visual quality and refined details. Even if prior information is utilized, there can exist multiple high-resolution solutions for the same low-resolution image. Thus, the inverse problem of super-resolution satisfies our central assumption \eqref{eq:centralassumption}. Assessing the quality of the output is not straightforward, there are many different quantitative metrics - such as the peak signal to noise ratio (PSNR) or the structural similarity index (SSIM) - and these metrics only loosely correlate to human perception \cite{anwar2020deep}. Another key challenge is to determine the limit to the super-resolution factor. This is difficult in super-resolution of multi-spectral satellite data due to dataset sizes and the resulting need for quantitative methods. A resulting challenge is to choose the optimal fraction between high and low resolution for different sensors and selections of spectral bands. This question is unanswered. In \cite{atkinson1997choosing}, the question of what the best spatial resolution for obtaining the optimum accuracy on a down-stream task is, such as land cover classification, is tackled. The authors \cite{atkinson1997choosing} claim that the spatial resolution with the maximal local variance is optimal. The underlying assumption is that scenes are assumed to be composed of objects with different reflectance values on a continuous background and that if the spatial resolution is that of the objects the local variance will be maximal \cite{curran1987airborne}. In \cite{mercer1911experimental}, as early as 1911, and \cite{smith1938empirical} in 1938, the authors empirically investigate the relation between local variance and spatial resolution. Theorem \ref{lem:lemmasym} provides a lower bound depending on the super-resolution factor to the super-resolution error of any method on a given dataset. The best super-resolution methods on a given dataset can be determined with the upper bound of Theorem \ref{lem:lemmasym}. In the following, we only compute the RMSE to validate the accuracy bounds provided by Theorem \ref{lem:lemmasym}.

\subsubsection{Spatial super-resolution of multi-spectral satellite data as an inverse problem}

\begin{figure}[ht]
    \centering
    \begin{subfigure}{0.3\textwidth}
        \centering
        \includegraphics[width=\textwidth]{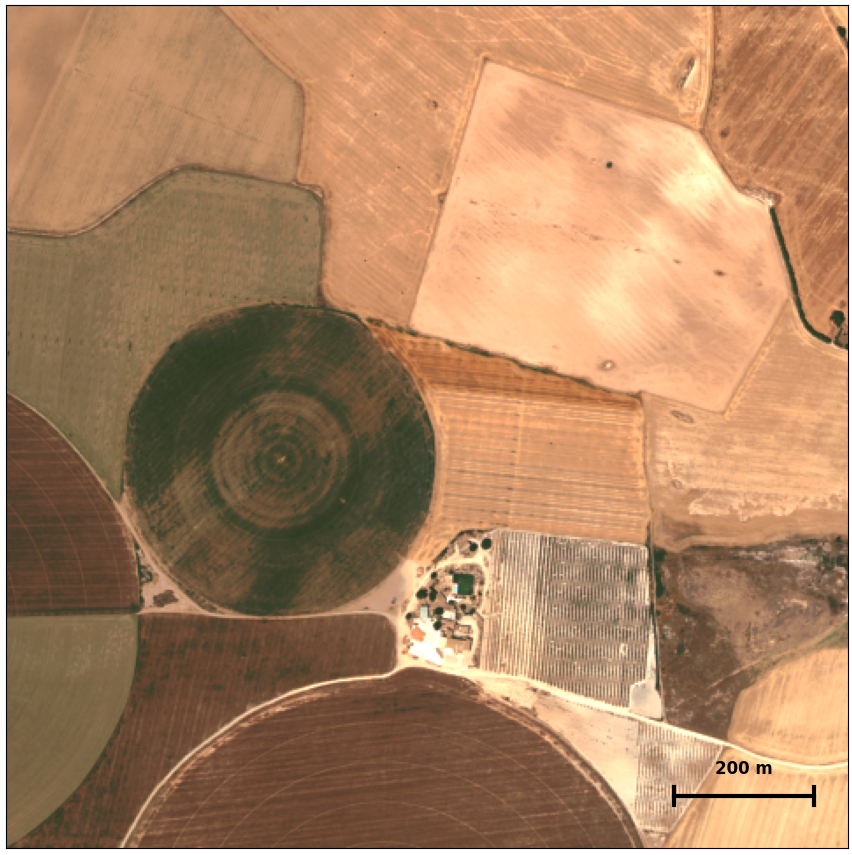}
        \caption{High Resolution}
    \end{subfigure}\hfill
    \begin{subfigure}{0.3\textwidth}
        \centering
        \includegraphics[width=\textwidth]{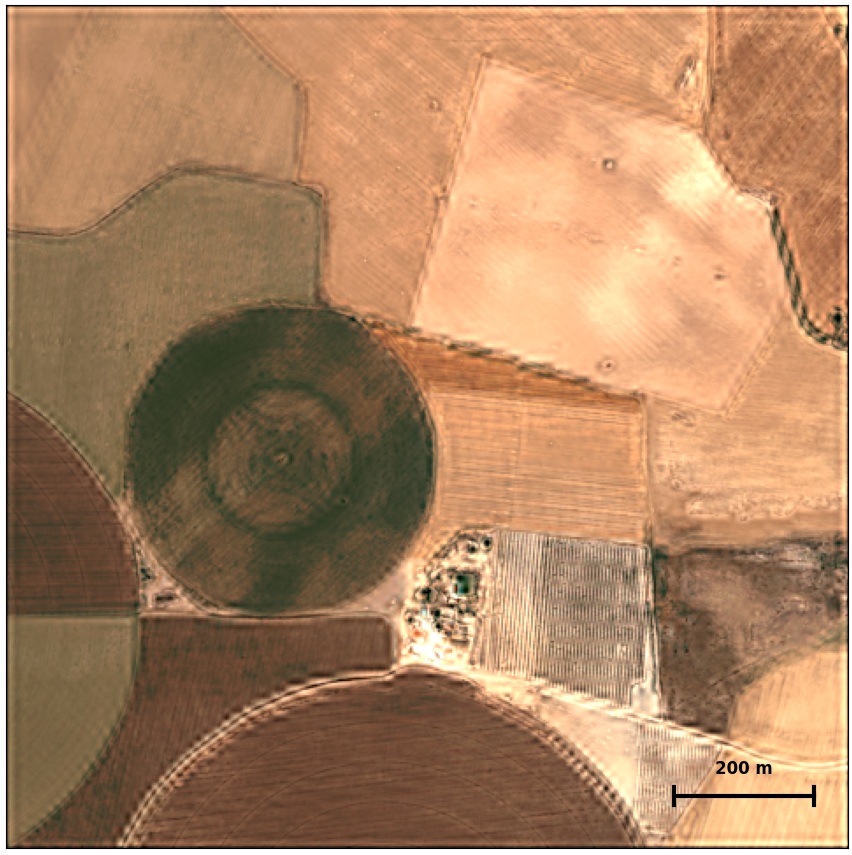}
        \caption{Reflected Version}
    \end{subfigure}\hfill
    \begin{subfigure}{0.3\textwidth}
        \centering
        \includegraphics[width=\textwidth]{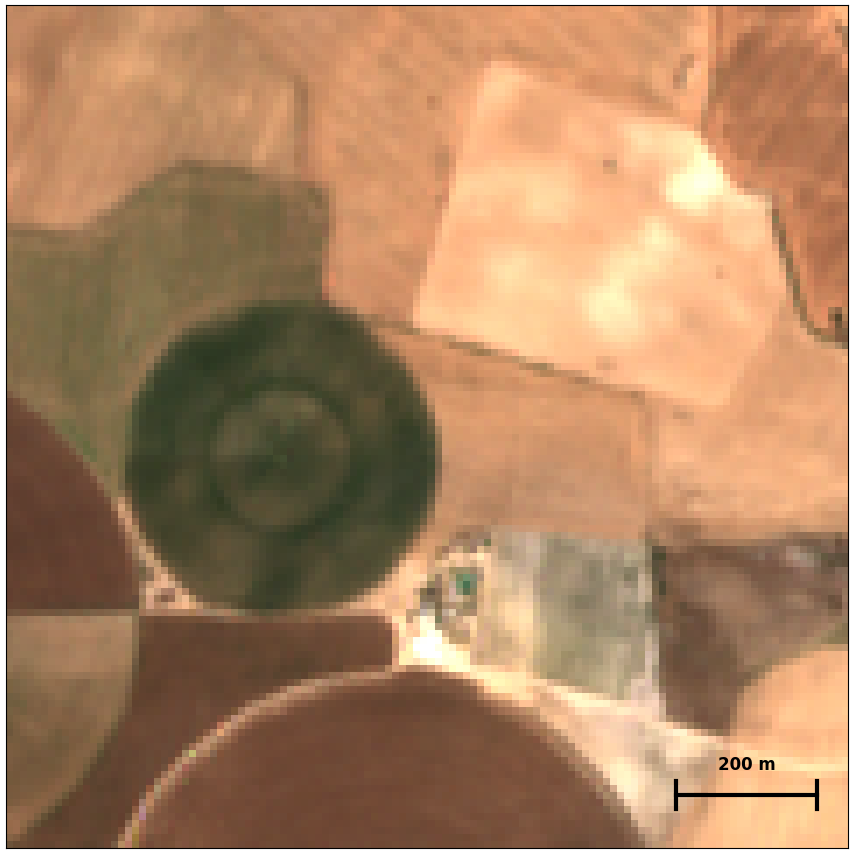}
        \caption{Low Resolution}
    \end{subfigure}
    \vspace{-5pt}
    \caption{RGB example visualizations of different versions of a multi-spectral satellite Croplands image from the NAIP dataset.
    The images correspond to (a) the original high-resolution image, (b) a high-resolution version reflected around the orthogonal component of the down-sampling map's kernel (according to \ref{eq:reflected_null}), and (c) the low-resolution image computed with the down-sampling operator. 
    }\label{fig:lemsymvalid}
    \vspace{-5pt}
\end{figure}

In super-resolution of multi-spectral satellite data the aim is to enhance spatial resolution of multi-dimensional arrays encoding the reflectance values of different spectral bands at possibly different spatial resolutions. For fourfold super-resolution, the task is to obtain a higher resolution satellite image $x \in [0,r_{max}]^{d^s \times 4d^x_1 \times 4d^y_1}$ where $d^s \in \mathbb{N}$ is the number of spectral bands, $d^x_1,d^y_1 \in \mathbb{N}$ are the number of pixels, and $r_{max} >0$ is the maximum reflectance value, from a lower resolution image $y \in [0,r_{max}]^{d^s \times d^x_1 \times d^y_1}$. We consider the simplest possible forward model that band-wise maps the higher resolution to lower resolution data with bilinear interpolation and additive noise $y_s = DS(x_s)+e_s$, for $s \in \{1, \dots, d^s\}$ where $e_s \in \mathcal{B}(0,\mathds{1}\epsilon) \subset \mathbb{R}^{d^x_1\times d^y_1}$, $x_s \in [0,r_{max}]^{ 4d^x_1 \times 4d^y_1}$ and $y_s \in [0,r_{max}+\epsilon]^{ d^x_1 \times d^y_1}$. For this down-sampling, each pixel in the output only depends on its neighboring pixels within a distance of two pixels. Hence, the set $\mathcal{M}_1 = [0,r_{max}]^{d^s \times 4d^x_1 \times 4d^y_1}$ and the forward model are given by
\begin{align}
\begin{split}
F:  \mathcal{M}_1\times \mathcal{B}(0,\mathds{1}\epsilon) \rightarrow  [0,r_{max}+\epsilon]^{d^s \times d^x_1 \times d^y_1},
(x,e) \mapsto F(x,e) = DS(x)+e.
\end{split}
\end{align}
As the inverse problem of super-resolution satisfies our central assumption \eqref{eq:centralassumption}, we validate the lower bound Theorem \ref{lem:lemmasym} by computing the average symmetric kernel size and the loss in the symmetrized dataset.

\subsubsection{Numerical experiments}

\begin{figure}
\begin{subfigure}{0.5\linewidth}
\begin{tikzpicture}[scale=0.75]
	\begin{axis}[legend pos=north west,%
		xmax=250, xmin=35, ymax=500, ymin=35, 
		title={\large \textbf{Diffusion (OpenSR)}},
		xlabel={$\frac{1}{2}\operatorname{SKersize}(F,\mathcal{M}_1,\mathcal{E},2)_3$},
		ylabel={RMSE},
        scatter/classes={%
		Forest={mark=pentagon*,green},%
        Cropland={mark=triangle*,cyan},
        Urban={mark=otimes*,magenta},
        Others={mark=star,darkgray}}%
	]
	\addplot[scatter,only marks, scatter src=explicit symbolic]%
	table[meta=label]{
       LB       Loss          label
118.776195 145.660616         Others
100.006142 120.225405         Others
 91.496606 120.193823         Others
240.898933 304.997557         Others
189.684670 251.136039         Others
108.569042 135.888302         Others
178.122174 212.396914         Others
152.557799 194.873139         Others
 85.977114 110.928741         Others
 88.185266 116.604161         Others
 58.527757  76.994133         Others
 85.735821 110.164691         Others
 56.523410  75.720327         Others
175.551602 219.993541         Others
242.515715 302.915154         Others
 38.882530  57.188992         Others
240.557530 291.083078         Others
167.389799 219.036430         Others
 46.881412  65.297628         Others
110.184685 132.084443         Others
145.997537 176.845585         Others
165.395087 198.866157         Others
108.240378 136.786986         Others
286.815001 367.034187         Others
 72.982413  98.217417         Others
 75.083039 101.320676         Others
 65.413678  88.108411         Others
 69.696176  91.140514         Others
 47.772405  66.708656         Others
204.872401 265.472280         Others
136.795184 175.101374         Others
173.073231 212.652431         Others
171.940190 213.182425         Others
304.416480 402.805783         Others
 65.700522  90.784001         Others
129.760583 161.142936         Others
134.981144 177.669098         Others
 54.081988  76.428172         Others
214.407902 272.366507         Others
105.835308 134.299477         Others
 77.598541  99.104460         Others
155.862804 196.444574         Others
 68.997226 100.322755         Others
167.816763 211.278817         Others
 47.036873  72.756614         Others
 73.208827  96.030075         Others
 79.407167 107.344393         Others
 57.582421  78.076579         Others
 94.557467 117.989619         Others
163.382727 204.875725         Others
112.295893 165.721966         Others
 50.880598  67.726197         Others
177.399702 195.450820         Others
 82.259886 100.806937         Others
190.636313 234.538209         Others
 86.359829 119.966266         Others
176.921216 226.718000         Others
186.322404 239.882608         Others
219.675988 276.317429         Others
 66.074697  87.587763         Others
 74.982719 103.053805         Others
108.687637 135.864990         Others
209.188231 259.367170         Others
152.976863 194.354899         Others
 61.082422  81.114746         Others
142.322298 177.235026         Others
111.497350 139.629329         Others
 85.814564 115.078895         Others
 57.731778  77.202309         Others
117.433654 140.774125         Others
147.999350 182.669169         Others
 63.332780  81.800722         Others
 39.274972  53.905652         Others
100.599824 134.031513       Cropland
145.030489 185.106451       Cropland
 56.601254  83.789155       Cropland
141.021908 200.245735       Cropland
 88.240810 115.444779       Cropland
113.018706 134.503804       Cropland
132.476743 166.838586       Cropland
 82.606945 109.120829       Cropland
188.535344 207.733978       Cropland
 54.112550  73.998871       Cropland
155.338855 199.199123       Cropland
 72.502713  98.135470       Cropland
 72.102781  99.060175       Cropland
155.743386 201.372079       Cropland
151.871341 197.289584       Cropland
 87.139443 111.870148       Cropland
155.520907 206.328144       Cropland
115.380789 137.533554       Cropland
155.577315 193.565854       Cropland
321.386267 437.516332          Urban
235.967037 294.857851          Urban
240.502203 305.305646          Urban
207.139647 261.726858          Urban
209.342528 265.978396          Urban
258.290263 320.505935          Urban
258.932385 327.520081          Urban
229.265782 298.630695          Urban
197.822177 246.434883          Urban
212.093846 259.001190          Urban
230.316734 290.682303          Urban
 54.812009  65.763938         Forest
 95.334090 116.075313         Forest
 67.164759  84.915185         Forest
 76.883237  90.710874         Forest
 81.188124  99.360228         Forest
 66.642685  81.129763         Forest
 90.683465 110.114391         Forest
 57.575692  72.706449         Forest
 77.173210  95.448274         Forest
100.838304 124.462494         Forest
	};
	\addplot[lightgray, thick, domain=0:300] (x,x);
	\addplot[lightgray, thick, domain=0:300] (x,2*x);
    \legend{Forest, Cropland, Urban, Others}
	\end{axis}
\end{tikzpicture}
\end{subfigure}
\hfill
\begin{subfigure}{0.48\linewidth}
\begin{tikzpicture}[scale=0.75]
	\begin{axis}[legend pos=north west,%
		xmax=250, xmin=35, ymax=500, ymin=35, 
		title={\large \textbf{Bilinear Interpolation}},
		xlabel={$\frac{1}{2}\operatorname{SKersize}(F,\mathcal{M}_1,\mathcal{E},2)_3$},
        scatter/classes={%
		Forest={mark=pentagon*,green},%
        Cropland={mark=triangle*,cyan},
        Urban={mark=otimes*,magenta},
        Others={mark=star,darkgray}}%
	]
	\addplot[scatter,only marks, scatter src=explicit symbolic]%
	table[meta=label]{
       LB       Loss          label
108.569042 124.194065         Others
152.557799 174.025519         Others
 85.977114  99.984144         Others
 88.185266 108.558874         Others
 58.527757  67.819137         Others
240.898933 281.151331         Others
 85.735821  99.176491         Others
175.551602 207.678398         Others
242.515715 279.200687         Others
 38.882530  47.625707         Others
240.557530 281.765832         Others
167.389799 194.512760         Others
 46.881412  56.318289         Others
165.395087 184.617574         Others
108.240378 124.884172         Others
286.815001 335.595657         Others
 72.982413  84.386428         Others
 75.083039  90.390958         Others
167.816763 200.202568         Others
 47.036873  57.908550         Others
 65.413678  78.192276         Others
 69.696176  83.847804         Others
 47.772405  57.406278         Others
204.872401 239.084824         Others
171.940190 203.976198         Others
304.416480 352.315914         Others
 65.700522  86.529473         Others
129.760583 149.923074         Others
 54.081988  64.834162         Others
214.407902 253.520507         Others
105.835308 126.056470         Others
 77.598541  85.790021         Others
155.862804 182.738081         Others
 68.997226  84.426953         Others
 73.208827  86.147567         Others
118.776195 135.212820         Others
100.006142 113.992485         Others
 91.496606 112.650552         Others
189.684670 215.775068         Others
178.122174 199.991184         Others
 56.523410  67.506450         Others
145.997537 164.914127         Others
110.184685 120.259235         Others
136.795184 156.531992         Others
173.073231 204.907300         Others
134.981144 162.078190         Others
 94.557467 112.905370         Others
 50.880598  62.812887         Others
 66.074697  81.090285         Others
 61.082422  72.297151         Others
111.497350 124.289172         Others
 57.731778  69.071394         Others 
 79.407167  96.054927         Others
 57.582421  69.806044         Others
163.382727 188.037329         Others
112.295893 143.363621         Others
177.399702 187.422901         Others
 82.259886  94.716472         Others
190.636313 210.727406         Others
 86.359829 104.039142         Others
176.921216 205.670995         Others
186.322404 222.542555         Others
219.675988 251.350089         Others
 74.982719  91.655332         Others
108.687637 121.938173         Others
209.188231 236.630182         Others
152.976863 181.920651         Others
142.322298 164.215014         Others
 85.814564 101.764364         Others
117.433654 133.883957         Others
147.999350 176.313032         Others
 63.332780  76.790585         Others
 39.274972  49.004403         Others
 72.102781  85.125673       Cropland
100.599824 122.537125       Cropland
145.030489 165.118335       Cropland
 56.601254  68.364994       Cropland
 88.240810 106.934463       Cropland
113.018706 124.253749       Cropland
141.021908 171.107515       Cropland
132.476743 151.319008       Cropland
155.743386 185.256735       Cropland
151.871341 177.291241       Cropland
 87.139443 100.972337       Cropland
155.520907 183.195792       Cropland
115.380789 127.671767       Cropland
 82.606945  96.649588       Cropland
188.535344 197.435404       Cropland
 54.112550  67.745898       Cropland
155.338855 180.825034       Cropland
 72.502713  86.988040       Cropland 
155.577315 180.078886       Cropland
 95.334090 108.653237         Forest
 54.812009  61.134464         Forest
 57.575692  65.270550         Forest
100.838304 112.932586         Forest
 66.642685  74.774727         Forest
 67.164759  75.697242         Forest
 76.883237  84.927679         Forest
 90.683465 100.577141         Forest
 81.188124  91.428014         Forest
 77.173210  87.419460         Forest
235.967037 279.071687          Urban
321.386267 381.975837          Urban
209.342528 252.533954          Urban
229.265782 275.399112          Urban
207.139647 241.728584 	       Urban
240.502203 289.456526          Urban
258.290263 300.093734          Urban
258.932385 296.458376          Urban
197.822177 227.490419          Urban
212.093846 246.655430          Urban
230.316734 272.571128          Urban
	};
	\addplot[lightgray, thick, domain=0:300] (x,x);
	\addplot[lightgray, thick, domain=0:300] (x,2*x);
      \legend{Forest, Cropland, Urban, Others}
	\end{axis}
\end{tikzpicture}
\end{subfigure}
\vspace{5pt}\\
\begin{subfigure}{0.5\linewidth}
\begin{tikzpicture}[scale=0.75]
	\begin{axis}[legend pos=north west,%
		xmax=250, xmin=35, ymax=500, ymin=35, 
		title={\large\textbf{Bicubic Interpolation}},
		xlabel={$\frac{1}{2}\operatorname{SKersize}(F,\mathcal{M}_1,\mathcal{E},2)_3$},
		ylabel={RMSE},
        scatter/classes={%
		Forest={mark=pentagon*,green},%
        Cropland={mark=triangle*,cyan},
        Urban={mark=otimes*,magenta},
        Others={mark=star,darkgray}}%
	]
	\addplot[scatter,only marks, scatter src=explicit symbolic]%
table[meta=label]{
       LB       Loss          label    
 73.208827  82.200890         Others
 79.407167  91.151658         Others
 57.582421  66.302714         Others
163.382727 181.069861         Others
112.295893 133.510337         Others
 54.081988  62.017190         Others
214.407902 241.056509         Others
105.835308 119.754004         Others
 77.598541  83.365460         Others
155.862804 174.429256         Others
 68.997226  79.831327         Others
167.816763 189.872204         Others
 47.036873  54.534067         Others
177.399702 184.707479         Others
 82.259886  91.173949         Others
190.636313 204.659904         Others
 86.359829  98.466002         Others
176.921216 196.733883         Others
186.322404 211.610762         Others
219.675988 241.905097         Others
 74.982719  86.507005         Others
108.687637 118.615064         Others
209.188231 228.348040         Others
152.976863 172.615089         Others
175.551602 197.458427         Others
242.515715 267.448494         Others
 38.882530  45.599256         Others
240.557530 268.886865         Others
167.389799 186.257388         Others
 46.881412  54.240680         Others
 85.735821  95.079974         Others
165.395087 179.042141         Others
108.240378 120.307229         Others
286.815001 320.510945         Others
 72.982413  81.179610         Others
 75.083039  85.788007         Others
 65.413678  74.320007         Others
 69.696176  79.478867         Others
 47.772405  54.960163         Others
204.872401 228.677401         Others
171.940190 193.966816         Others
304.416480 336.578069         Others
 65.700522  79.737738         Others
129.760583 144.741756         Others
142.322298 157.924769         Others
 85.814564  97.606769         Others
117.433654 130.005316         Others
147.999350 167.098465         Others
 63.332780  72.577912         Others
 39.274972  46.845833         Others
 58.527757  65.228945         Others
152.557799 167.749960         Others
 85.977114  95.568483         Others
 88.185266 103.117468         Others
108.569042 120.010467         Others
240.898933 268.297844         Others
136.795184 150.832950         Others
173.073231 195.068548         Others
145.997537 159.468848         Others
 56.523410  64.502966         Others
110.184685 118.159406         Others
 91.496606 105.764200         Others
118.776195 130.892530         Others
100.006142 110.585734         Others
189.684670 207.732662         Others
134.981144 153.938611         Others
178.122174 193.670193         Others
 57.731778  66.111155         Others
 94.557467 107.677955         Others
111.497350 121.268285         Others
 50.880598  59.801589         Others
 66.074697  76.932252         Others
 61.082422  69.351628         Others
 72.102781  81.798982       Cropland
 88.240810 101.656269       Cropland
113.018706 121.844590       Cropland
155.520907 175.195760       Cropland
141.021908 161.735087       Cropland
115.380789 124.807875       Cropland
132.476743 145.834179       Cropland
 54.112550  63.801497       Cropland
155.338855 172.918219       Cropland
188.535344 195.473588       Cropland
 82.606945  92.483541       Cropland
151.871341 169.535644       Cropland
 100.599824 115.753491      Cropland
145.030489 159.407763       Cropland
 56.601254  65.091888       Cropland
 72.502713  82.595779       Cropland
155.743386 175.988999       Cropland
 87.139443  97.160474       Cropland
 54.812009  59.933917       Forest
 67.164759  73.596273       Forest
 57.575692  63.493185       Forest
 90.683465  97.978837       Forest
  95.334090 104.816478      Forest
 66.642685  72.838795       Forest
 76.883237  82.987991       Forest
 77.173210  84.705743       Forest
100.838304 109.568404       Forest
 81.188124  88.743814       Forest
321.386267 361.719273       Urban
209.342528 238.714061       Urban
235.967037 265.317362       Urban
229.265782 260.694959       Urban
197.822177 218.291440       Urban
230.316734 259.231363       Urban
212.093846 236.173486       Urban
155.577315 172.498353       Cropland
240.502203 273.857303       Urban
258.932385 284.775297       Urban
258.290263 286.875876       Urban
207.139647 230.802442       Urban
};
	\addplot[lightgray, thick, domain=0:300] (x,x);
	\addplot[lightgray, thick, domain=0:300] (x,2*x);
    \legend{Forest, Cropland, Urban, Others}
	\end{axis}
\end{tikzpicture}
\end{subfigure}
\hspace{10pt}
\begin{minipage}{0.44\linewidth}
    \captionsetup{width=\linewidth}\vspace{-6.2cm}
    \captionof{table}{RMSE for different methods on the two datasets and the corresponding lower accuracy bound - half of the average symmetric kernel size.    
    }\label{tab:losssrnaip}
    \begin{tabular}{l|c|c}
    \hline
    Dataset Size & $119$ & $6 \times 119$ \\
    \hline
    Diffusion & $165.87$ & $184.07$  \\
    Bilinear & $169.00$ & $168.72$ \\
    Bicubic & $160.43$  & $161.43$ \\
     \hline
    \hline
    $\frac{1}{2}\operatorname{SKersize}$  & & $144.92$ \\
    \hline
    \end{tabular}
\end{minipage}
\caption{RMSE of different super-resolution methods against the average symmetric kernel size for triplets, $M=3$, of $119$ of images. The lower and upper gray lines ($y=x$ and $y=2x$) visualize the lower and upper accuracy bounds.  All RMSE values are within the computed accuracy bounds. The separation into urban, cropland and forest images highlight the increase of the kernel size with increasing detail content of the images. The columns in Table \ref{tab:losssrnaip} show the RMSE for different super-resolution methods on the dataset with $M = 119$ and with $2M = 6\times119$ images. The output of Algorithm \ref{alg:feasset}, with feasible sets of size \(N(K) = 3\), is used in addition to Theorem \ref{lem:lemmasym} to obtain the enlarged dataset. }
\label{fig:estimatorvskerfscm_eo}
\vspace{-5pt}
\end{figure}

\noindent We compute the average symmetric kernel size from Theorem \ref{lem:lemmasym}. Here we use the Moore-Penrose inverse to compute the projection onto the kernel of the down-sampling forward model, see Section \textbf{SM} $5$. In addition to our library, we use the computational framework to evaluate methods for super-resolution of multi-spectral satellite data from \cite{SR2024Ces}, see Section \textbf{SM} $4.3$ for further details. We use four different multi-spectral satellite image datasets from the \href{https://huggingface.co/datasets/isp-uv-es/opensr-test}{opensr-test} 
benchmark dataset collection: one from the National Agriculture Imagery Program (NAIP), one from SPOT imagery that was obtained from the worldstat dataset and the SPAIN- URBAN and CROPS datasets. The datasets consist of high-resolution aerial imagery at $4.5$ m ground sampling distance captured in the visible and near-infrared spectrum (RGBNIR) and all Sentinel-2 L1C and L2A spectral bands. The satellite images cover diverse land cover classes ranging from urban areas over crop fields to forests. We compute the RMSE loss $\mathcal{L}(\mathcal{D}_M,\phi,2)$ from \eqref{eq:empiricalerror} for  different super-resolution maps to validate the lower accuracy bound of Theorem \ref{lem:lemmasym}. We use the approximate inverse map introduced by \cite{TrustworthwSR2025}, referred to as Diffusion (Open SR) map in the following and bilinear upscaling and bicubic interpolations from the pytorch library \cite{paszke2019pytorch}. The scatter plots of the RMSE against half the average symmetric kernel size in Figure \ref{fig:estimatorvskerfscm_eo} validate the lower bound of Theorem \ref{lem:lemmasym} for each super-resolution map, each low-resolution measurement and land cover class. The assumption that the satellite image reflected around the orthogonal component of the kernel of the forward model is a realistic image in Theorem \ref{lem:lemmasym} is visually validated in Figure \ref{fig:lemsymvalid}. All considered super-resolution maps are consistently optimal and within the accuracy bounds of Theorem \ref{lem:lemmasym}. We compute the RMSE on the whole input and output datasets of Algorithm \ref{alg:feasset} in Table \ref{tab:losssrnaip}. Table \ref{tab:losssrnaip} shows that the RMSE remains consistently in the same range for different land cover classes and super-resolution methods across the input and output datasets. In Section \textbf{SM} $4.4$ we show that the RMSE on the dataset of the selected methods is in the range of the RMSE obtained in previous work. By computing the average symmetric kernel size, we can provide a theoretically sound and empirically validated bound to modeling accuracy. These accuracy bounds, can help determine the optimal spatial and temporal resolution of earth observation images, which are bottlenecks in downstream applications \cite{benediktsson2012very,quattrochi1997scale,chen2023large,kong2023super}. 
\section{Convergence under additional assumptions}\label{sec:additionalassconv}

To obtain convergence guarantees for the average kernel size and the reconstruction error in \eqref{eq:empiricalerror} we need to make additional assumptions to the problem set-up. We now assume that
\begin{enumerate}
    \item the metrics $d_\mathcal{X}$, $d_\mathcal{Y}$, $d_\mathcal{Z}$ are chosen so that the topologies induced by the metrics are second countable, thus, they admit a countable base.
    \item the metrics $d_\mathcal{X}$ and $d_\mathcal{Z}$ satisfy the Heine-Borel property, i.e., that all closed and bounded sets are compact. 
    \item $(\mathcal{X}\times\mathcal{Z}, \mathcal{B}(\mathcal{X}\times\mathcal{Z}))$, $(\mathcal{M}_1\times\mathcal{E}, \mathcal{B}(\mathcal{M}_1\times\mathcal{E}))$, $(\mathcal{M}_2,\mathcal{B}(\mathcal{M}_2))$ are measurable spaces and that the forward model $F\colon \mathcal{M}_1\times\mathcal{E} \to \mathcal{M}_2$ is Borel measurable. We equip  $\mathcal{X}\times\mathcal{Z}$ with a finite Borel measure $\mu$ supported on $\mathcal{M}_1\times\mathcal{E}$, such that $\mu(\mathcal{M}_1\times\mathcal{E})>0$. We equip $\mathcal{M}_2$ with the pushforward measure $F_*\mu$ given by $(F_*\mu)(E) = \mu(F^{-1}(E))$ for every $E \in \mathcal{B}(\mathcal{M}_2)$. 
    \item that $\mathcal{M}_1$ is compact and $\mathcal{E}$ is a Borel set.
\end{enumerate}

\noindent Under these assumptions, the disintegration of measure $\mu^y$ of $\mu$ with respect to $F$ is a probability measure for $F_*\mu$-a.e. $y \in \mathcal{M}_2$ and exists, see  \cite[Proposition 7.1]{gottschling2023existence}. With the disintegration of the measure $\mu$, we can define the average kernel size on measurable metric spaces from \cite[Definition 3.8]{gottschling2023existence},
\begin{align}\label{eq:kersizegudd}
		\operatorname{kersize}^\text{a}&(F,\mathcal{M}_1, \mathcal{E}, p)^p\nonumber\\
        &= \int_{\mathcal{M}_2} \int_{F^{-1}(y)}  \int_{F^{-1}(y)} d_{\mathcal{X}}(x,x')^p \, d\mu^y(x,e) \ d\mu^y(x',e') \ d(F_*\mu)(y).
\end{align}
By approximating the inner, thereafter the outer integrals of the average kernel size from \cite[Definition 3.8]{gottschling2023existence}, with Monte Carlo integration, Theorem \ref{thm:convavgalgnon} provides guarantees for the convergence of the average kernel size in \eqref{eq:lowerboundogboiz} to a normalization of the average kernel size on measurable metric spaces. 

\begin{theorem}[Convergence of approximation to the average kernel size]\label{thm:convavgalgnon}
Let $p \in [1, \infty)$, $K \in \mathbb{N}$, $\{y_k\}_{k=1}^K \subset \mathcal{M}_2$ be i.i.d. samples from $\mu_F := \tfrac{1}{F_*\mu(\mathcal{M}_2)} F_*\mu$, and $\{(x_{k,n},e_{k,n})\}_{n=1}^{N(K)}$ be i.i.d. samples from $\mu^{y_k}$, such that $N(k)=N(K)$ for all $k \in \{1, \ldots,K\}$. Let $N(K) \to \infty$ whenever $K \to \infty$. Then, for $\delta > 0$, there is $c(K,\delta)$ such that with probability greater than or equal to $1-c(K,\delta)$ we have that 
    \begin{equation}\label{eq:convergence}
    \left|\operatorname{Kersize}(F,\mathcal{M}_1,\mathcal{E},p)_K^p- \tfrac{1}{(F_*\mu)(\mathcal{M}_2)}\operatorname{kersize}^\text{a}(F, \mathcal{M}_1, \mathcal{E},p)^p\right| \leq \delta,
\end{equation}
where $c(K,\delta) \to 0$ as $K \to \infty$. 
\end{theorem}

\noindent For fixed $\delta>0$ and $\eta\in(0,1)$, to guarantee that \eqref{eq:convergence} holds with probability at least $1-\eta$, it suffices to choose the number of outer samples $K$ and inner samples $N(K)$ such that
\[
K \ge \tfrac{2v_{\max}^2}{\delta^2}\log\!\left(\tfrac{4}{\eta}\right),
\qquad
N(K)
\ge \tfrac{2v_{\max}^2}{\delta^2}\log\!\left(\tfrac{4K}{\eta}\right),
\]
where $v_{\max} = \sup_{y \in \mathcal{M}_2}\sup_{x,x' \in F_y} d_{\mathcal{X}}(x,x')^p$. For single-valued measurable approximate inverse maps $\phi: \mathcal{M}_2 \rightarrow \mathcal{X}$, Lemma \ref{lem:avererrconv} shows that the empirical reconstruction error converges to the generalization error \cite{mohri2018foundations}, given by
\begin{align*}
\mathbb{E}_{\tilde{\mu}}\left(d_{\mathcal{X}}\Bigl(x,\phi(F(x,e))\Bigr)^p\right)=\int_{\mathcal{M}_1\times\mathcal{E}} d_{\mathcal{X}}\Bigl(x,\phi(F(x,e))\Bigr)^p
\,d\tilde{\mu}(x,e),
\end{align*}
\noindent with high probability as the number of samples from $\tilde{\mu}:=\frac{1}{\mu(\mathcal{M}_1\times\mathcal{E})}\mu$ increases.

\begin{lemma}[Convergence of the empirical approximation to the  generalization error]\label{lem:avererrconv}
Let $p \in [1,\infty)$, $\{(x_m,e_m)\}_{m=1}^M$ be i.i.d. samples drawn from the probability measure
$\tilde{\mu}$ on $\mathcal{M}_1\times\mathcal{E}$. Let $\phi: \mathcal{M}_2 \rightarrow \mathcal{X}$ be measurable and such that $u_{\max}:=\sup_{(x,e)\in \mathcal{M}_1\times\mathcal{E}}d_{\mathcal{X}}\Bigl(x,\phi(F(x,e))\Bigr)^p < \infty$.
Then, for sufficiently small $\delta>0$, there is $c(M,\delta)>0$, such that with probability greater than or equal to $1-c(M,\delta)$, we have that
\begin{align}
\Bigg|\tfrac{1}{M}\sum_{m=1}^Md_{\mathcal{X}}\Bigl(x_m,\phi(F(x_m,e_m))\Bigr)^p-\mathbb{E}_{\tilde{\mu}}\left(d_{\mathcal{X}}\Bigl(x,\phi(F(x,e))\Bigr)^p\right)\Bigg|\le \delta,
\end{align}
where $c(M,\delta)=2\exp\left(-\tfrac{2M\delta^2}{u_{\max}^2}\right) \to0$ as $M\to\infty$.
\end{lemma}

\noindent As a consequence of Theorem \ref{thm:convavgalgnon}, Lemma \ref{lem:avererrconv}, Theorem \ref{thm:lowerboundapprx} and \cite[Theorem 3.9]{gottschling2023existence} we can obtain the following diagram, Figure \ref{fig:comaprisondiag}, for an optimal measurable approximate inverse map $\phi: \mathcal{M}_2 \rightarrow \mathcal{X}$.

\begin{figure}[h!]
\centering
\begin{tikzpicture}[scale=0.85, transform shape, node distance=0.9cm,
    every node/.style={fill=white}, align=left,
	base/.style = {rectangle, draw=gray,
                           text centered, font=\small}  
    ]
      
      \node [base, rounded corners, xshift=-8.1cm] (sth) {$\delta>0$ distance\\ probability $\geq 1-c$};

      \node [base, rounded corners,  right of=sth, xshift = 5.3cm, yshift=0cm] (sthel) {$\delta>0$ distance \\ probability $\geq 1-c$};

       \node [base, rounded corners,  right of=sth, xshift = 10.7cm, yshift=0cm] (sthell) {$\delta>0$ distance \\ probability $\geq 1-c$};
      
      \node [base, above of=sth, xshift = 0cm, yshift=1.5cm] (fm) {$\operatorname{Kersize}(F,\mathcal{M}_1,\mathcal{E},p)_K$};
      
     \node [base, above of=sthel, xshift = 0cm, yshift=1.5cm] (fm2) {$\mathcal{L}(\mathcal{D}_M, \phi,p)$};

      \node [base, above of=sthell, xshift = 0cm, yshift=1.5cm] (fm3) {$\operatorname{Kersize}(F,\mathcal{M}_1,\mathcal{E},p)_K$};

      \node [base, below of=sth, yshift=-1.5cm] (im) {$\operatorname{kersize}^\text{a}(F, \mathcal{M}_1, \mathcal{E},p)$};
      
       \node [base, below of=sthel, yshift=-1.5cm] (im2) {$ \left(\int_{\mathcal{M}_1\times\mathcal{E}} d_{\mathcal{X}}\Bigl(x,\phi(F(x,e))\Bigr)^p
\,d\mu(x,e)\right)^{1/p}$};

 \node [base, below of=sthell, yshift=-1.5cm] (im3) {$\operatorname{kersize}^\text{a}(F, \mathcal{M}_1, \mathcal{E},p)$};

    \draw[thick, ->,color=black]  (fm) -- node {$\cdot \tfrac{1}{2} \leq $}(fm2);
    \draw[thick, ->,color=black]  (fm2) -- node {$ \leq $}(fm3);
	\draw[thick, ->,color=black]  (fm) -- node {\small{$\cdot (F_*\mu)(\mathcal{M}_2)^{1/p}$}}(sth);
	\draw[thick, ->,color=black]  (fm2) -- node {\small{$\cdot (F_*\mu)(\mathcal{M}_2)^{1/p}$}} (sthel);
    \draw[thick, ->,color=black]  (fm3) -- node {\small{$\cdot (F_*\mu)(\mathcal{M}_2)^{1/p}$}} (sthell);
	\draw[thick, ->,color=black]  (sthel) -- node {$c \to 0$ \\ $M \to \infty$} (im2);
	\draw[thick, ->,color=black]  (sth) -- node {$c \to 0$ \\ $K \to \infty$} (im);
    \draw[thick, ->,color=black]  (sthell) -- node {$c \to 0$ \\ $K \to \infty$} (im3);
    \draw[thick, ->,color=black]  (im) -- node {$\cdot \tfrac{1}{2} \leq $}(im2);
    \draw[thick, ->,color=black]  (im2) -- node {$ \leq $}(im3);
\end{tikzpicture}
\caption{Diagram under the additional assumptions in this section and under the conditions of Theorem \ref{thm:convavgalgnon} and Lemma \ref{lem:avererrconv}. By \cite[Theorem 3.9]{gottschling2023existence} the lower horizontal axis holds true. By Theorem \ref{thm:lowerboundapprx} for $M=\sum_{k=1}^KN(k)$ the upper horizontal axis holds true. Note that $(F_*\mu)(\mathcal{M}_2)=\mu(F^{-1}(\mathcal{M}_2))=\mu(\mathcal{M}_1\times\mathcal{E})$. And by Theorem \ref{thm:convavgalgnon}, the left and right, vertical axes hold true and by Lemma \ref{lem:avererrconv} the middle vertical axis.}\label{fig:comaprisondiag}
\end{figure}

\noindent Figure \ref{fig:comaprisondiag} provides a conceptual comparison between Theorem \ref{thm:lowerboundapprx} and \cite[Theorem 3.9]{gottschling2023existence}. A distinction between the two results is the amount of structure required for their proofs. Theorem \ref{thm:lowerboundapprx} does not require the existence of a disintegration of measure and therefore does not rely on assumptions $1$-$4$ of this section, whereas \cite[Theorem 3.9]{gottschling2023existence} depends on these assumptions. The convergence result for the average kernel size on measurable metric spaces requires equal sample numbers on the feasible sets. In contrast to \cite[Theorem 3.9]{gottschling2023existence}, Theorem \ref{thm:lowerboundapprx} is a geometric-analytic statement that holds for arbitrary feasible set sizes and becomes non-trivial solely through the non-injectivity of the forward model and does not rely on Monte Carlo sampling or probabilistic approximation arguments. Theorem \ref{thm:lowerboundapprx} preserves the inequalities for all samples, whereas with a Monte Carlo approximation for \cite[Theorem 3.9]{gottschling2023existence} this is not necessarily valid with probability one. The previous results also show that the finite-sample lower bound from Theorem \ref{thm:lowerboundapprx} imposes an asymptotically necessary probabilistic constraint on convergence of the empirical reconstruction error to the generalization error. 

\begin{corollary}[Asymptotically necessary condition for convergence in probability of the empirical error to the generalization error]
\label{thm:finitesampleconsistency}
Assume the hypotheses of Theorem~\ref{thm:convavgalgnon}
and Lemma~\ref{lem:avererrconv}. Let \begin{align*}
\mathcal{E}(\phi):=\mathbb{E}_{\tilde{\mu}}\left(d_{\mathcal X}\bigl(x,\phi(F(x,e))\bigr)^p\right)^{1/p}, \text{ and } \Gamma_K:=\frac12\operatorname{Kersize}(F,\mathcal M_1,\mathcal E,p)_K,
\end{align*}
for $K \in \mathbb{N}$. Assume that $\left|\mathcal{L}(\phi,\mathcal D_M,p)-\mathcal{E}(\phi)\right|\to 0$ in probability as $M\to\infty$. Then, for $\varepsilon>\delta>0$, we have that
\begin{align*}
\mathbb P\left(\mathcal L(\phi,\mathcal D_M,p)<\Gamma_K-\varepsilon
\right)\le\mathbb P\left(\left|\mathcal L(\phi,\mathcal D_M,p)-\mathcal E(\phi)
\right|>\varepsilon-\delta\right)+\tilde{c}(K,\delta),
\end{align*}
where $\tilde{c}(K,\delta) \to 0$ as $K\to \infty$. 
\end{corollary}

\noindent Corollary \ref{thm:finitesampleconsistency} is non-trivial for $M\to\infty$ and $K\to\infty$, but does not impose any other relation between the sample numbers $M$ and $K$. 

\section{Relation to previous work}

This work establishes computable lower and upper bounds for the lowest achievable \emph{average} reconstruction error of approximate inverse maps for nonlinear inverse problems. The framework applies to non-injective forward models, including ones with mixed additive and multiplicative noise, and yields finite-sample computable reconstruction guarantees. Related notions of optimal reconstruction have been studied in several settings. Generalized instance optimality and corresponding upper bounds for approximate inverse maps were investigated in \cite{fundamental14} under a generalized null space property (NSP). The generalized NSP in \cite{fundamental14} is equivalent to injectivity of the nonlinear forward model on the signal set $\mathcal{M}_1$, i.e.
$\pi_1\left(\mathcal{N}(F)\right)\cap(\mathcal{M}_1-\mathcal{M}_1)=\{0\}$, see \cite[Section 4]{gottschling2023existence}. Nonlinear and non-injective forward models on $\mathcal{M}_1$ are relevant in the examples considered in this work. A large body of work in optimal recovery theory studies lower and upper bounds for the \emph{worst-case} reconstruction error. Examples include Gelfand widths for noiseless linear inverse problems \cite{pinkus2012n}, best $k$-term approximation \cite{CoDaDe-08}, optimal recovery for set-valued approximate inverse maps on Banach spaces \cite{arestov1986optimal}, and extensions to metric spaces \cite{magaril1991optimal}. For noisy linear inverse problems, \cite{binev2022optimal} introduced the notion of optimal learning and characterized worst-case reconstruction accuracy via the Chebyshev radius. In the setting of linear forward models with additive noise, the Chebyshev radius is closely related to the diameter of the feasible sets considered in the present work. Similarly, \cite{cohen2022optimal} developed a framework for stable nonlinear approximation methods and corresponding worst-case reconstruction guarantees. See also \cite{micchelli1977survey} for an early survey of optimal recovery theory. The work of Traub et al.~\cite{traub1983information} considers the radius of information through a family of admissible approximation sets $S(f,\delta)$ that are nested. In Section \textbf{SM} $6$, we show in metric spaces under the assumptions of this work, that the local radius of information $r(F,x)$ at $x \in \mathcal{M}_1$ from \cite{traub1983information} satisfies $\tfrac{1}{2}\operatorname{diam}(F_y) \le r(F,x) \le \operatorname{diam}(F_y)$, where $y = F(x,0)$ and $\operatorname{diam}(F_y) = \sup_{x,x' \in F_y}d_{\mathcal{X}}(x,x')$. The present work formulates reconstruction in terms of metric and Hausdorff distances on feasible sets and optimal decoders, without requiring a nested family of approximation sets.  Worst-case and average-case reconstruction bounds for normed spaces were also studied in \cite{plaskota1996noisy}. Many of these works consider finite- or infinite-dimensional Banach spaces and derive bounds for worst-case reconstruction errors measured with respect to prescribed norms. The present work complements this literature by computable finite-sample bounds for average reconstruction error nonlinear applicable approximate inverse maps such as stochastic estimators and deep neural networks. This computational perspective is motivated by the gap between approximation-theoretic expressivity and practical reconstruction performance in scientific machine learning \cite{adcock2021gap}. Average-case reconstruction guarantees are relevant in modern data-driven inverse problems, where reconstruction quality is evaluated empirically over large datasets rather than through worst-case adversarial instances and offers a tractable approach to questions raised recently in \cite{burger2024learning}.

\section{Discussion and conclusion}

A limitation of our framework is that the bounds in Theorem~\ref{thm:lowerboundapprx} are evaluated on finite datasets produced by Algorithm~\ref{alg:feasset}. In Section \ref{sec:applications}, we verified the bounds numerically on fluorescence localization microscopy and multi-spectral satellite super-resolution.  However, to extend guarantees beyond a fixed dataset, one needs quantitative control of how the computable average kernel size behaves under increasing sample sizes. Under additional measurability and i.i.d.\ sampling assumptions, Theorem~\ref{thm:convavgalgnon} shows convergence (up to normalization) to the measure-theoretic average kernel size of \cite{gottschling2023existence}, and Corollary~\ref{thm:finitesampleconsistency} identifies half of the average kernel size computed with finite samples as an asymptotic accuracy barrier, whenever empirical errors converge in probability to generalization error. Extending these convergence results to dependent and more realistic sampling regimes, such as exchangeable or ergodic processes, is an important open problem. A second limitation is that Algorithm~\ref{alg:feasset} relies on access to feasible-set samples, which may be restrictive outside supervised settings. The validity of the finite-sample inequalities themselves is distribution-free, and in the linear additive-noise case Theorem~\ref{lem:lemmasym} yields an efficient bound even from a single feasible solution. A future direction could be to develop estimators of average kernel size that do not require ground-truth and measurement pairs, enabling guarantees for self-supervised inverse problems based solely on forward-model consistency.

\section*{Acknowledgments}
N.~M.~Gottschling, D.~Iagaru and J.~Gawlikowski would like to acknowledge the funding provided by the German Aerospace Center from the department EO Data Science at the Earth Observation Center. This work was supported by the Helmholtz Association's Initiative and Networking Fund on the HAICORE (at) FZJ partition. N.~M.~Gottschling was supported by the Laboratory Directed Research and Development Program of Oak Ridge National Laboratory, managed by UT-Battelle, LLC, for the U. S. Department of Energy.

{\footnotesize
\bibliographystyle{abbrv}
\bibliography{references}
}

\section{Proofs}
We now prove the main results.

\subsection{Proof of Theorem \ref{thm:lowerboundapprx}}

\begin{proof}[Proof of Theorem \ref{thm:lowerboundapprx}]
\noindent Let  $\{y_k\}_{k=1}^K \subset \mathcal{M}_2 = F(\mathcal{M}_1\times\mathcal{E})$. Now obtain $\left\{F_{y_k}^{N(k)}\right\}_{k=1}^K$ from Algorithm \ref{alg:feasset}. For each distinct $y_k \in \mathcal{M}_2$ we obtain a lower approximation to the feasible set $\pi_1(F^{-1}(y_k))$  defined in \eqref{eq:feasset} with $F_{y_k}^{N(k)}= \{x_{k,n}\}_{n=1}^{N(k)}$. Let $\phi: \mathcal{M}_2 \rightrightarrows \mathcal{X}$ be an approximate inverse map. Now using the inequality $(a+b)^p \leq 2^{p-1}(a^p+b^p)$ for $a,b \geq 0$ and $p \in [1,\infty)$, we obtain for fixed $k \in \{1, \dots, K\}$ with $y_k \in \mathcal{M}_2$, $z_k \in \phi(y_k)$ and fixed $n,n' \in \{1, \dots,N(k)\}$ with $n<n'$ and $x_{k,n}, x_{k,n'} \in F_{y_k}^{N(k)}$ that
\begin{align}\label{eq:triangle0}
    \begin{split}
     d_{\mathcal{X}}(x_{k,n}, x_{k,n'})^p &\leq (d_{\mathcal{X}}(x_{k,n},z_k )+d_{\mathcal{X}}(z_k, x_{k,n'}))^p \\
    &\leq 2^{p-1} \left(d_{\mathcal{X}}(x_{k,n},z_k)^p +d_{\mathcal{X}}(z_k,x_{k,n'})^p\right).
    \end{split}
\end{align}
\noindent Taking the supremum over $z_k \in \phi(y_k)$ yields,
\begin{align}\label{eq:triangle1}
    \begin{split}
     d_{\mathcal{X}}(x_{k,n}, x_{k,n'})^p \leq 2^{p-1} \left(d_{\mathcal{X}}^H(x_{k,n},\phi(y_k))^p +d_{\mathcal{X}}^H(\phi(y_k),x_{k,n'})^p\right).
    \end{split}
\end{align}
\noindent Note that $N(k) \ge 2$ and thus $N(k)-1\ge 1$. Now, summing over all $n,n' \in \{1, \dots,N(k)\}$ such that $n<n'$, and thus all $x_{k,n}, x_{k,n'} \in F_{y_k}^{N(k)}$, and multiplying by $\tfrac{2}{N(k)-1}$ yields
\begin{align}\label{eq:triangle2}
\begin{split}
    \frac{2}{N(k)-1} \sum_{1\le n <n' \le N(k)}& d_{\mathcal{X}}(x_{k,n},x_{k,n'})^p \\ &\leq  \frac{2^{p-1}2}{N(k)-1} \sum_{1\le n <n' \le N(k)}  \left(d_{\mathcal{X}}^H(x_{k,n},\phi(y_k))^p +d_{\mathcal{X}}^H(\phi(y_k), x_{k,n'})^p\right) \\
    &= \frac{2^{p-1}}{N(k)-1}\sum_{n=1}^{N(k)} \sum_{n'=1, n' \neq n}^{N(k)} \left(d_{\mathcal{X}}^H(x_{k,n},\phi(y_k))^p +d_{\mathcal{X}}^H(\phi(y_k), x_{k,n'})^p\right)\\ & = \frac{2^{p-1}}{N(k)-1}\sum_{n=1}^{N(k)} \sum_{n'=1}^{N(k)}\left(d_{\mathcal{X}}^H(x_{k,n},\phi(y_k))^p +d_{\mathcal{X}}^H(\phi(y_k), x_{k,n'})^p\right) \\ &- \frac{2^{p-1}}{N(k)-1} \sum_{n=1}^{N(k)}2 d_{\mathcal{X}}^H(x_{k,n},\phi(y_k))^p.
\end{split}
\end{align}
\noindent Now we use the fact that for a constant $c>0$, we have that $\frac{1}{N(k)} \sum_{n=1}^{N(k)} c = c$ and that $\phi(y_k)$ is independent of $n \in \{1, \dots, N(k)\}$ in RHS of the above inequality \eqref{eq:triangle2}, to obtain
\begin{align}
\begin{split}
  \frac{2}{N(k)-1} &\sum_{1\le n <n' \le N(k)}   d_{\mathcal{X}}(x_{k,n} , x_{k,n'})^p \\ & \le \frac{2^{p-1}}{N(k)-1}\sum_{n=1}^{N(k)} \sum_{n'=1}^{N(k)}\left(d_{\mathcal{X}}^H(x_{k,n},\phi(y_k))^p +d_{\mathcal{X}}^H(\phi(y_k), x_{k,n'})^p\right) 
  \\ 
  &- \frac{2^{p-1}}{N(k)-1} \sum_{n=1}^{N(k)}2 d_{\mathcal{X}}^H(x_{k,n},\phi(y_k))^p \\
    & =  \frac{2^{p-1}}{N(k)-1}2N(k)\sum_{n=1}^{N(k)}d_{\mathcal{X}}^H(x_{k,n},\phi(y_k))^p -  \frac{2^{p-1}}{N(k)-1}2\sum_{n=1}^{N(k)}d_{\mathcal{X}}^H(x_{k,n},\phi(y_k))^p \\ 
    &=  \frac{2^{p}(N(k)-1)}{N(k)-1}\sum_{n=1}^{N(k)}d_{\mathcal{X}}^H(x_{k,n},\phi(y_k))^p = 2^{p}\sum_{n=1}^{N(k)}d_{\mathcal{X}}^H(x_{k,n},\phi(y_k))^p.
\end{split}
\end{align}
\noindent Now we sum over all $k \in \{1, \dots, K\}$ and divide by $2^pM=2^p \sum_{k=1}^KN(k)$, to obtain
\begin{align}\label{eq:boundeasy}
    \frac{1}{2^pM} \sum_{k=1}^K   \frac{2}{N(k)-1} \sum_{1\le n <n' \le N(k)}  d_{\mathcal{X}}(x_{k,n}, x_{k,n'})^p \leq  \frac{1}{M} \sum_{k=1}^K \sum_{n=1}^{N(k)} d_{\mathcal{X}}^H(x_{k,n},\phi(y_k))^p.
\end{align}
\noindent In the following we prove for $M = \sum_{k=1}^K N(k)$, $\left\{F_{y_k}^{N(k)}\right\}_{k=1}^K$ and $\mathcal{D}_M = \{(x_m,y_m)\}_{m=1}^M$ from Algorithm \ref{alg:feasset} and for any $\phi: \mathcal{M}_2 \rightrightarrows \mathcal{X}$, that the right hand side of \eqref{eq:boundeasy} satisfies
\begin{align}\label{eq:importantinequality}
  \frac{1}{M} \sum_{k=1}^K \sum_{n=1}^{N(k)} d_{\mathcal{X}}^H(x_{k,n} ,\phi(y_k))^p = \mathcal{L}(\mathcal{D}_M,\phi,p)^p = \frac{1}{M} \sum_{m=1}^M  d_{\mathcal{X}}^H(x_m,\phi(y_m))^p.
\end{align}

\noindent To prove this, we introduce two maps to reassign the summation indices. We define a surjective function that maps the indices in $\{1, \dots , M\}$ to the index of the corresponding feasible set by
\begin{align}\label{eq:mapindiexdown}
\begin{split}
    g: \{1, \dots,M\} & \rightarrow \{1, \dots,K\} \\
    m & \mapsto g(m)= \argmin_{k\in \{1,\dots,K\}} \left(2-\mathds{1}_{0 < m - \sum_{j=0}^{k-1} N(j)}(k)-\mathds{1}_{0\leq\sum_{j=1}^{k} N(j)-m}(k)\right),
\end{split}
\end{align}
\noindent where $N(0)=0$. For any $m \in \{1, \dots,M\}$ we have that $g(m)= k \in \{1, \dots,K\}$. Moreover, as by definition $M=\sum_{k=1}^K N(k)$, we have that $m$ and $k=g(m)$ are such that 
\begin{align}\label{eq:range}
    0 \leq \sum_{j=0}^{k-1} N(j) < m \leq \sum_{j=1}^{k} N(j) \leq M.
\end{align}
\noindent From \eqref{eq:range} for any $k \in \{1, \dots,K\}$ we have that 
\begin{align}
\emptyset \neq g^{-1}(k) = \left(\sum_{j=0}^{k-1} N(j), \sum_{j=1}^{k} N(j)\right]\cap\mathbb{N} \subset \{1, \dots, M\}.
\end{align}
Hence, $g$ is surjective. Now we show that for $k' \neq k$, $g^{-1}(k')\cap g^{-1}(k) = \emptyset$. Let $k' \neq k$, without loss of generality we have that $k < k'$ and $k'=k+1$. Then, we have that $(\sum_{j=0}^{k-1} N(j), \sum_{j=1}^{k} N(j)]\cap (\sum_{j=0}^{k} N(j), \sum_{j=1}^{k+1} N(j)] = \emptyset$ and, thus $g^{-1}(k')\cap g^{-1}(k) = \emptyset$. Now we define a bijective function depending on the function $g$ to reassign the indices of samples to samples allocated to the feasible sets by
\begin{align}\label{eq:mapindices2}
\begin{split}
    u: \{1, \dots,M\}  & \rightarrow \bigcup_{k=1}^K \{k\} \times \{1, \dots, N(k)\}  \\
    m &\mapsto u(m) = \left(g(m), m - \sum_{k=1}^{g(m)-1}N(k)\right).
\end{split}
\end{align}
\noindent Note that $u:  \{1, \dots,M\} \rightarrow \bigcup_{k=1}^K \{k\} \times \{1, \dots, N(k)\}$ is surjective, as the function $g$ defined in \eqref{eq:mapindiexdown} is surjective and as, by definition, we have that
\begin{align*}
\{1, &\dots, N(g(m))\} \\
&=\left\{n\in \mathbb{N}: \quad n=m' - \sum_{k=1}^{g(m)-1}N(k) \text{ for } m' \in \left(\sum_{k=0}^{g(m)-1} N(k), \sum_{k=1}^{g(m)} N(k)\right]\cap\mathbb{N}\right\}.
\end{align*}
\noindent Moreover, the function $u$ is injective, as for $m \neq m' \in \{1, \dots,M\}$, $u(m) \neq u(m')$. Now as $u:  \{1, \dots,M\}  \rightarrow \bigcup_{k=1}^K \{k\} \times \{1, \dots, N(k)\}$ is bijective, we can reassign the indices in the summation over $m \in \{1, \dots,M\}$ with $u$. Then, we obtain
\begin{align}\label{eq:loss23}
\begin{split}
M\mathcal{L}(\mathcal{D}_M,\phi,p)^p &= \sum_{m=1}^M d_{\mathcal{X}}^H(x_m,\phi(y_m))^p = \sum_{m=1}^M  d_{\mathcal{X}}^H(x_{u(m)},\phi(y_{u(m)}))^p \\
&= \sum_{k \in \{1, \dots,K\}} \sum_{n=1}^{N(k)} d_{\mathcal{X}}^H(x_{k,n},\phi(F(x_{k,n},e_{k,n})))^p.
\end{split}
\end{align}
\noindent Dividing by $M= \sum_{k=1}^K N(k)$ yields \eqref{eq:importantinequality}. Combining \eqref{eq:boundeasy} and \eqref{eq:importantinequality} and taking the $p$-th root yields the desired inequality. \newline

\noindent Now define for each $k$
$$\Theta_k := \operatorname*{arg\,min}_{z\in\mathcal{X}} \sum_{n=1}^{N(k)} d_{\mathcal{X}}(x_{k,n},z)^p,$$ which is nonempty, which we prove below. Now we define a map 
\begin{align*}\label{eq:optmap}
\theta: \mathcal{M}_2 \rightrightarrows \mathcal{X},
\end{align*}
\noindent by
$$\theta(y)=
\begin{cases}
\Theta_k, & y=y_k \text{ for some }k\in\{1,\dots,K\},\\
x_0, & y \in \mathcal{M}_2\setminus\{y_1,\dots,y_K\},
\end{cases}$$ 
\noindent for an arbitrary fixed $x_0 \in \mathcal{X}$. Then $\theta$ attains the infimum in $\inf_{\phi:\mathcal{M}_2\rightrightarrows\mathcal{X}}\mathcal{L}(\mathcal{D}_M,\phi,p)$. For fixed $k \in \{1, \dots, K\}$ now define 
\begin{align*}
f_k: \mathcal{X} & \rightarrow [0,\infty) \\
z & \mapsto f_k(z):= \frac{1}{N(k)}\sum_{n=1}^{N(k)}d_{\mathcal{X}}(x_{k,n},z)^p.
\end{align*}
\noindent We define
\begin{align*}
r_k^p = \max_{n,n' \in \{1, \dots, N(k)\}} d_{\mathcal{X}}(x_{k,n},x_{k,n'})^p.
\end{align*}
\noindent In the following we prove that 
\begin{enumerate}
\item[$(I)$] $f_k$ is continuous,
\item[$(II)$] $\argmin_{z \in \mathcal{X}} f_k(z) = \argmin_{z \in \mathcal{B}(x_{k,n},2r_k)} f_k(z)$ for all $n \in \{1, \dots, N(k)\}$.
\end{enumerate}
\noindent In order to prove $(I)$ - that $f_k$ is continuous - we first prove that $f_k^{1/p}$ is continuous. By Minkowski's inequality for the counting measure and $p \in [1,\infty)$, we obtain
\begin{align}
\begin{split}
f_k(z)^{1/p} &= \left(\frac{1}{N(k)} \sum_{n=1}^{N(k)}d_{\mathcal{X}}(x_{k,n},z)^p\right)^{1/p} \\
&\leq \left( \frac{1}{N(k)}\sum_{n=1}^{N(k)}d_{\mathcal{X}}(x_{k,n},z')^p\right)^{1/p} + \left( \frac{1}{N(k)}\sum_{n=1}^{N(k)}d_{\mathcal{X}}( z ',z)^p\right)^{1/p} \\
& \leq f_k(z')^{1/p} + d_{\mathcal{X}}( z ',z).
\end{split}
\end{align}
\noindent Reversing the role of $z,z' \in \mathcal{X}$ and subtracting the resulting inequalities leads to
\begin{align}
|f_k(z)^{1/p}-f_k(z')^{1/p}| \leq d_{\mathcal{X}}(z ',z).
\end{align}
\noindent Thus, $f_k^{1/p}$ is continuous and, hence, $f_k = (f_k^{1/p})^p$ is continuous. Now we proceed to prove $(II)$. Let $n \in \{1, \dots, N(k)\}$. Fix $z \in \mathcal{X} \setminus \mathcal{B}(x_{k,n},2r_k)$. Then, for $n' \in \{1, \dots, N(k)\}$ by the reverse triangle inequality we have that
\begin{align}
d_{\mathcal{X}}(x_{k,n'},z) \geq d_{\mathcal{X}}(z,x_{k,n})-d_{\mathcal{X}}(x_{k,n},x_{k,n'}) > 2r_k -r_k =r_k.
\end{align}
\noindent Summing over $n' \in \{1, \dots, N(k)\}$ and dividing by $N(k)$, proves that $f_k(z) \geq r_k^p$ for $z \in \mathcal{X}\setminus \mathcal{B}(x_{k,n},2r_k)$. On the other hand, $f_k(x_{k,n}) \leq r_k^p$, which proves that $f_k$ cannot have minimizers outside the ball $\mathcal{B}(x_{k,n},2r_k)$ for any $n \in \{1, \dots, N(k)\}$. This proves $(II)$. Now recall the that $\theta(y_k) = \argmin_{z \in \mathcal{B}(x_{k,n},2r_k)}f_k(z)$ for any $n \in \{1, \dots,N(k)\}$. As the ball in the quotient space $\mathcal{B}([x_{k,n}],2r_k)\subset \mathcal{X}/\sim$ is closed and bounded with respect to the pseudo metric $d_{\mathcal{X}}$ on $\mathcal{X}$, it is compact in the metric space on the quotient set $\mathcal{X}/\sim$ by Heine-Borel. Thus, as $f_k$ is continuous by the Extreme Value Theorem there exists a minimum that is attained on $\mathcal{B}([x_{k,n}],2r_k)\subset \mathcal{X}/\sim$. Hence, the map $\theta$ defined in \eqref{eq:optmap} has non-empty values. We now prove that $\theta$ is an optimal map. Let $\phi: \mathcal{M}_2 \rightrightarrows \mathcal{X}$ be a map, then by the minimizing definition of $\theta$ and as dividing by $N(k)$ does not change the minimizing property, we have that
\begin{align*}
 \sum_{n=1}^{N(k)}d_{\mathcal{X}}^H(x_{k,n},\theta(F(x_{k,n},e_{k,n})))^p \leq  \sum_{n=1}^{N(k)}d_{\mathcal{X}}^H(x_{k,n},\phi(F(x_{k,n},e_{k,n})))^p.
\end{align*}
\noindent Summing over $k \in \{1, \dots,K\}$ and dividing by $M = \sum_{k=1}^KN(k)$ yields,
\begin{align*}
\frac{1}{M}\sum_{k=1}^K \sum_{n=1}^{N(k)}&d_{\mathcal{X}}^H(x_{k,n},\theta(F(x_{k,n},e_{k,n})))^p \\ 
&\leq \frac{1}{M}\sum_{k=1}^K\sum_{n=1}^{N(k)}d_{\mathcal{X}}^H(x_{k,n},\phi(F(x_{k,n},e_{k,n})))^p.
\end{align*}
\noindent Now using the bijective index map $u: \{1, \dots,M\} \rightarrow \bigcup_{k=1}^K \{k\} \times \{1, \dots, N(k)\}$, taking the $p$-th root and, then, taking the infimum over maps $\phi: \mathcal{M}_2 \rightrightarrows \mathcal{X}$ yields
\begin{align*}
&\left(\frac{1}{M}\sum_{m=1}^M d_{\mathcal{X}}^H(x_{m},\theta(F(x_m,e_m)))^p\right)^{1/p} \\
&\leq \inf_{\phi: \mathcal{M}_2 \rightrightarrows \mathcal{X}}\left(\frac{1}{M}\sum_{m=1}^Md_{\mathcal{X}}^H(x_{m},\phi(F(x_{m},e_{m})))^p\right)^{1/p} = \inf_{\phi: \mathcal{M}_2 \rightrightarrows \mathcal{X}} \mathcal{L}(\mathcal{D}_M, \phi,p).
\end{align*}
\noindent The reverse inequality trivially holds and thus
\begin{align*}
\left(\frac{1}{M}\sum_{m=1}^Md_{\mathcal{X}}^H(x_{m},\theta(F(x_m,e_m)))^p\right)^{1/p} =\inf_{\phi: \mathcal{M}_2 \rightrightarrows \mathcal{X}} \mathcal{L}(\mathcal{D}_M, \phi,p).
\end{align*}
\noindent By the minimization property of $\theta$ we have that for any $n' \in \{1, \dots, N(k)\}$ that
\begin{align*}
\sum_{n=1}^{N(k)}d_{\mathcal{X}}^H(x_{k,n},\theta(F(x_{k,n},e_{k,n})))^p \leq \sum_{n=1}^{N(k)}d_{\mathcal{X}}(x_{k,n},x_{k,n'})^p.
\end{align*}
\noindent Summing over $n' \in \{1, \dots, N(k)\}$, such that $n'\neq n$ and dividing by $N(k)-1$ yields that
\begin{align*}
 \sum_{n=1}^{N(k)}d_{\mathcal{X}}^H(x_{k,n},\theta(F(x_{k,n},e_{k,n})))^p \leq \frac{1}{N(k)-1} \sum_{n=1}^{N(k)}\sum_{n'=1, n' \neq n}^{N(k)}d_{\mathcal{X}}(x_{k,n},x_{k,n'})^p.
\end{align*}
\noindent Summing over $k \in \{1, \dots,K\}$, dividing by $M$ and using the bijective index map $u$ on the left hand side of the inequality yields that
\begin{align*}
\frac{1}{M}\sum_{m=1}^Md_{\mathcal{X}}^H(x_{m},\theta(F(x_m,e_m)))^p \leq \frac{1}{M}\sum_{k=1}^K \frac{2}{N(k)-1} \sum_{1 \le n < n'\le N(k)}d_{\mathcal{X}}(x_{k,n},x_{k,n'})^p.
\end{align*}
\noindent Thus, taking the $p$-th root and using the definition of the average kernel size and the above, we obtain
\begin{align*}
			\frac{1}{2}\operatorname{Kersize}(F,\mathcal{M}_1,\mathcal{E},p)_K
			 \leq  \inf_{\phi: \mathcal{M}_2 \rightrightarrows \mathcal{X}} \mathcal{L}(\mathcal{D}_M, \phi,p)\leq \operatorname{Kersize}(F,\mathcal{M}_1,\mathcal{E},p)_K.
\end{align*}
\noindent In Section \textbf{SM} $1$, we show that the lower accuracy bound is sharp. In particular, we show with a linear algebra example that for a specific forward model, dataset and set of noise as well as approximate inverse map the lower accuracy bound is attained. 
\end{proof}

\subsection{Proof of Theorem \ref{lem:lemmasym}}

\begin{proof}
\noindent The proof of Theorem \ref{lem:lemmasym} closely follows the proof of Theorem \ref{thm:lowerboundapprx}. We start with proving part $(i)$. We merely change the application of the triangle inequality in \eqref{eq:triangle1} to obtain the desired result. In particular, using the inequality $\|x+x'\|^p \leq 2^{p-1}(\|x\|^p+\|x'\|^p)$, we obtain for fixed $m \in \{1, \dots, M'\}$ with $y_m \in \mathcal{M}_2$. Let $e_{m} \in \mathcal{E}$ be given by $e_{m}= y_m-F(x_{m},0)$ and note that $y_m = F(x_{m},e_{m})$. Now define $v_{m} = \pi_1(P_{\mathcal{N}(F)}(x_{m},e_{m}))$ and $v_{m}^{\perp} = \pi_1(P_{\mathcal{N}^{\perp}(F)}(x_{m},e_{m}))$ and by the properties of the Moore-Penrose inverse, see Section \textbf{SM} $4$, we have that $x_{m'}= v_{m}^{\perp}-v_{m} \in \pi_1(F^{-1}(y_m))$ with $m'= m'(m)=M'+m$. Then, we have that
\begin{align}\label{eq:triangle100}
    \begin{split}
      \| 2 v_{m}\|^p &=  \|x_{m} -\phi(y_m) +\phi(y_m) - (v_{m}^{\perp}-v_{m})\|^p  \\
      &\leq 2^{p-1} \left(\|x_{m} -\phi(y_m)\|^p +\|\phi(y_m) - x_{m'}\|^p\right).
      \end{split}
\end{align}
\noindent Averaging over $m \in \{1, \dots,M'\}$ yields
\begin{align}\label{eq:triangle10}
     \frac{2^p}{M'} \sum_{m=1}^{M'} \| v_{m}\|^p  \leq  \frac{2^{p-1}}{M'} \sum_{m=1}^{M'} \left(\|x_{m} -\phi(y_m)\|^p +\|\phi(y_m) - x_{m'(m)}\|^p\right).
\end{align}
\noindent Reassigning the indices with $m': \{1, \dots, M'\} \rightarrow \{1, \dots,M\}$ where $m'(m)=M'+m$ yields
\begin{align}\label{eq:triangle101}
     \frac{2^p}{M'} \sum_{m=1}^{M'} \| v_{m}\|^p  \leq  \frac{2^{p}}{M} \sum_{m=1}^{M} \|x_{m} -\phi(y_m)\|^p.
\end{align}
\noindent To prove part $(ii)$ the analogous steps as in the proof of Theorem \ref{thm:lowerboundapprx} can be used. Here the approximate feasible sets need to be replaced by $\{x_m, x_m-2P_{\mathcal{N}(F)}x_m\}$ for each $m \in \{1, \dots,M'\}$.
\end{proof}

\subsection{Proofs of Theorem \ref{thm:convavgalgnon}}

\begin{proof}[Proof of Theorem \ref{thm:convavgalgnon}]
For $F_*\mu$-a.e. $y \in \mathcal{M}_2$ define the probability measure $\tilde{\mu}^{y} = \mu^y \times \mu^y$. Then, the inner integrals of the average kernel size are given by
\begin{align}\label{eq:kersizegudd2}
		\operatorname{kersize}^\text{a}(F,\mathcal{M}_1, \mathcal{E}, p)
        &= \Bigg( \int_{\mathcal{M}_2} \int_{F^{-1}(y)}  \int_{F^{-1}(y)} d_{\mathcal{X}}(x,x')^p \, d\tilde{\mu}^y((x,e),(x',e')) \ d(F_*\mu)(y) \Bigg)^\frac{1}{p} \nonumber\\  
        &= \Bigg( \int_{\mathcal{M}_2} \mathbb{E}_{\tilde{\mu}^y}\left(d_{\mathcal{X}}(x,x')^p\right)  \ d(F_*\mu)(y) \Bigg)^\frac{1}{p}.
\end{align}
\noindent Now we transform the measure $F_*\mu$ of the outer integral into a probability measure by factoring out $0<(F_*\mu)(\mathcal{M}_2) <\infty$, leading to $\mu_F := \tfrac{1}{F_*\mu(\mathcal{M}_2)} F_*\mu$, to obtain
\begin{align*}
		\operatorname{kersize}^\text{a}(F,\mathcal{M}_1, \mathcal{E}, p)^p
        &= (F_*\mu)(\mathcal{M}_2) \int_{\mathcal{M}_2} \mathbb{E}_{\tilde{\mu}^y}\left(d_{\mathcal{X}}(x,x')^p\right)  \ \tfrac{1}{F_*\mu(\mathcal{M}_2)} d(F_*\mu)(y) \nonumber\\ 
        &= (F_*\mu)(\mathcal{M}_2) \mathbb{E}_{\mu_F}\left(\mathbb{E}_{\tilde{\mu}^y}\left(d_{\mathcal{X}}(x,x')^p\right)\right).
\end{align*}
\noindent Fix $\delta = \delta'/2 >0$. Define $v_{\mathrm{max}} = \sup_{y \in \mathcal{M}_2}\sup_{x,x' \in F_y} d_{\mathcal{X}}(x,x')^p$. Since, $\mathcal{M}_1$ is compact, its diameter with respect to $d_{\mathcal{X}}$ is finite and thus $v_{\mathrm{max}}<\infty$. As $\{y_k\}_{k=1}^K \subset \mathcal{M}_2$ are sampled independently with respect to $\mu_F$ and as $\tilde{\mu}^{y_k}$ is a probability measure, the expectations $\mathbb{E}_{\tilde{\mu}^{y_k}}(d_{\mathcal{X}}(x,x')^p)$ are independent random variables that are bounded by $v_{\mathrm{max}}$. Now we use Hoeffding's inequality \cite{hoeffding1963probability} for $\delta >0$ and $K \in \mathbb{N}$,
\begin{align}\label{eq:convcltbig1}
    \mathbb{P}\left( \left|\tfrac{1}{K} \sum_{k=1}^K  \mathbb{E}_{\tilde{\mu}^{y_k}}(d_{\mathcal{X}}(x,x')^p)  -  \mathbb{E}_{\mu_F}\left(\mathbb{E}_{\tilde{\mu}^{y}}(d_{\mathcal{X}}(x,x')^p)\right)\right| \geq \delta \right) \leq 2\exp\left(-2K\tfrac{\delta^2}{v_{\mathrm{max}}^2}\right).
\end{align}
\noindent Now let $\delta>0$ and fix $k\in\{1,\ldots,K\}$. Define $h((x,e),(x',e')):=d_{\mathcal{X}}(x,x')^p$.
Since $0\leq h((x,e),(x',e')) \leq v_{\mathrm{max}}$, $h$ is bounded and symmetric. Moreover, following \cite[Section 5q]{hoeffding1963probability} we have that
\begin{align*}
\tfrac{2}{N(K)(N(K)-1)}\sum_{1\le n<n'\le N(K)}d_{\mathcal{X}}(x_{k,n},x_{k,n'})^p=\tfrac{1}{N(K)(N(K)-1)} \sum_{n \neq n'}d_{\mathcal{X}}(x_{k,n},x_{k,n'})^p,
\end{align*}
is a $U$-statistic with $r=2$ associated to the bounded symmetric kernel $h$. Hence, as $\{(x_{k,n},e_{k,n})\}_{n=1}^{N(K)}$ are sampled independently with respect to the disintegration of measure $\mu^{y_k}$, by Hoeffding's inequality for bounded $U$-statistics \cite[Section 5a]{hoeffding1963probability} we obtain,
\begin{align*}
\mathbb{P}\Bigg(
\Bigg|\tfrac{1}{N(K)(N(K)-1)}\sum_{n \neq n'} d_{\mathcal{X}}(x_{k,n},x_{k,n'})^p&-\mathbb{E}_{\tilde{\mu}^{y_k}}
\big(d_{\mathcal{X}}(x,x')^p\big)\Bigg|\geq \delta
\Bigg)
\\
&\leq2\exp\left(-2\Big\lfloor\tfrac{N(K)}{2}\Big\rfloor\tfrac{\delta^2}{v_{\mathrm{max}}^2}
\right).
\end{align*}
\noindent Now define the events
\begin{align*}
A_k := &\left\{\left|\tfrac{1}{N(K)(N(K)-1)}\sum_{n \neq n'} d_{\mathcal{X}}(x_{k,n},x_{k,n'})^p - \mathbb{E}_{\tilde{\mu}^{y_k}}(d_{\mathcal{X}}(x,x')^p) \right|\geq \delta\right\} \\
B_K :=& \left\{\left|\tfrac{1}{K} \sum_{k=1}^K  \mathbb{E}_{\tilde{\mu}^{y_k}}(d_{\mathcal{X}}(x,x')^p)  -  \mathbb{E}_{\mu_F}\left(\mathbb{E}_{\tilde{\mu}^{y}}(d_{\mathcal{X}}(x,x')^p)\right)\right| \geq \delta\right\}
\end{align*}
Then, by the union bound we have that
\begin{align}\label{eq:thismaybeok2}
    \mathbb{P}\Bigg(\left(\bigcup_{k=1}^KA_k\right) \bigcup B_K\Bigg)\leq K 2\exp\left(-2\Big\lfloor\tfrac{N(K)}{2}\Big\rfloor\tfrac{\delta^2}{v_{\mathrm{max}}^2}\right) + 2\exp\left(-2K\tfrac{\delta^2}{v_{\mathrm{max}}^2}\right)
\end{align}
Equivalently using de Morgan's laws this yields,
\begin{align}\label{eq:convclt1big1}
    \mathbb{P}\Bigg(\left(\bigcap_{k=1}^KA_k^C\right) \bigcap B_K^C\Bigg)
    &\geq 1- \left(K 2\exp\left(-2\Big\lfloor\tfrac{N(K)}{2}\Big\rfloor\tfrac{\delta^2}{v_{\mathrm{max}}^2}\right) + 2\exp\left(-2K\tfrac{\delta^2}{v_{\mathrm{max}}^2}\right)\right)\nonumber\\  &= 1 - \tilde{c}(K, \delta).
\end{align}
Note that as $N(K)$ is increasing with $K$, we have that $\tilde{c}(K, \delta) \to 0$ as $K \to \infty$. Thus, $1 - \tilde{c}(K, \delta)>0$ for $K\in \mathbb{N}$ sufficiently large. With probability greater than or equal to $ 1-\tilde{c}(K, \delta)$ we have at the same time that for all $k = 1, \ldots K$ that
\begin{align}\label{eq:coolstuff10}
   \left|\tfrac{1}{N(K)(N(K)-1)}\sum_{n \neq n'} d_{\mathcal{X}}(x_{k,n},x_{k,n'})^p -  \mathbb{E}_{\tilde{\mu}^{y_k}}(d_{\mathcal{X}}(x,x')^p)\right| \leq \delta,
\end{align}
and that
\begin{align}\label{eq:coolstuff20}
    \left|\tfrac{1}{K} \sum_{k=1}^K  \mathbb{E}_{\tilde{\mu}^{y_k}}(d_{\mathcal{X}}(x,x')^p)  - \mathbb{E}_{\mu_F}\left(\mathbb{E}_{\tilde{\mu}^{y}}(d_{\mathcal{X}}(x,x')^p)\right) \right| \leq \delta.
\end{align}
As \eqref{eq:coolstuff10} and \eqref{eq:coolstuff20} hold at the same time with probability greater than or equal to $1-\tilde{c}(K, \delta)$, for each $k \in \{1, ...K\}$ we can use the triangle inequality, to obtain
\begin{align}\label{eq:bound001}
        \left|\tfrac{1}{K} \sum_{k=1}^K   \tfrac{1}{N(K)(N(K)-1)}\sum_{n\neq n'} d_{\mathcal{X}}(x_{k,n},x_{k,n'})^p - \mathbb{E}_{\mu_F}\left(\mathbb{E}_{\tilde{\mu}^{y}}(d_{\mathcal{X}}(x,x')^p)\right) \right|\leq 2\delta
\end{align}
with probability greater than or equal to $ 1-\tilde{c}(K, \delta)$. Now as $\delta = \delta'/2$, we obtain 
\begin{equation}
   \left|\operatorname{Kersize}(F,\mathcal{M}_1,\mathcal{E},p)_K^p- \tfrac{1}{(F_*\mu)(\mathcal{M}_2)}\operatorname{kersize}^\text{a}(F, \mathcal{M}_1, \mathcal{E},p)^p\right| \leq \delta',
\end{equation}
with probability greater than or equal to $1-c(K, \delta')=1-\tilde{c}(K, \delta'/2)$.

\noindent To obtain explicit sample complexity bounds, fix $\delta'>0$ and $\eta\in(0,1)$ and require that the total error probability satisfies $c(K,\delta')\le \eta$. From
\begin{align*}
c(K,\delta')
= 2K \exp\!\left(-\Big\lfloor \tfrac{N(K)}{2}\Big\rfloor \tfrac{\delta'^2}{2v_{\max}^2}\right)
+ 2\exp\!\left(-K\tfrac{\delta'^2}{2v_{\max}^2}\right),
\end{align*}
we bound each term by $\eta/2$. First, we have that $2\exp\!\left(-K\tfrac{\delta'^2}{2v_{\max}^2}\right)\le \tfrac{\eta}{2}$ 
is equivalent to $K \ge \tfrac{2v_{\max}^2}{\delta'^2}\log\!\left(\tfrac{4}{\eta}\right)$. Next, $2K \exp\!\left(-\Big\lfloor \tfrac{N(K)}{2}\Big\rfloor \tfrac{\delta'^2}{2v_{\max}^2}\right)\le \tfrac{\eta}{2}$
is equivalent to $\exp\!\left(-\Big\lfloor \tfrac{N(K)}{2}\Big\rfloor \tfrac{\delta'^2}{2v_{\max}^2}\right)
\le \tfrac{\eta}{4K}$. Taking logarithms yields $-\Big\lfloor \tfrac{N(K)}{2}\Big\rfloor \tfrac{\delta'^2}{2v_{\max}^2}
\le \log\!\left(\tfrac{\eta}{4K}\right)$, and hence $N(K)
\ge \tfrac{2v_{\max}^2}{\delta'^2}\log\!\left(\tfrac{4K}{\eta}\right)$.
Under these conditions, the approximation error is bounded by $\delta'$ with probability at least $1-\eta$.
\end{proof}

\begin{proof}[Proof of Lemma \ref{lem:avererrconv}]
Define $Z(x,e):=d_{\mathcal{X}}\Bigl(x,\phi(F(x,e))\Bigr)^p$, which is a random variable, as a composition of measurable functions. Since $\tilde{\mu}$ is a probability measure, we may write
\begin{align}
\int_{\mathcal{M}_1\times\mathcal{E}} d_{\mathcal{X}}\Bigl(x,\phi(F(x,e))\Bigr)^p
\,d\tilde{\mu}(x,e) = \mathbb{E}_{\tilde{\mu}}[Z].
\end{align}
Moreover, by definition of $u_{\max}$, we have that $0\le Z(x,e)\le u_{\max}$ for all $(x,e)\in\mathcal{M}_1\times\mathcal{E}$. Since $\{(x_m,e_m)\}_{m=1}^M$ are i.i.d. samples from $\tilde{\mu}$, the random variables $Z_m:=Z(x_m,e_m)$ for $m \in \{1,\ldots, M\}$ are i.i.d., as sample-wise transformations of i.i.d. variables, and bounded in $[0,u_{\max}]$. Applying Hoeffding's inequality yields
\begin{align}
\mathbb{P}\left(\left|\tfrac{1}{M}\sum_{m=1}^M Z_m-\mathbb{E}_{\tilde{\mu}}[Z]\right|
\ge \delta\right)\le2\exp\left(-\tfrac{2M\delta^2}{u_{\max}^2}\right).
\end{align}
Rewriting this inequality gives
\begin{align}
\mathbb{P}\Bigg(\Bigg|\tfrac{1}{M}\sum_{m=1}^Md_{\mathcal{X}}\Bigl(x_m,\phi(F(x_m,e_m))\Bigr)^p-\mathbb{E}_{\tilde{\mu}}[Z]\Bigg|\le \delta\Bigg)\ge 1-c(M,\delta),
\end{align}
with $c(M,\delta)=2\exp\left(-\tfrac{2M\delta^2}{u_{\max}^2}\right)$ which proves the claim.
\end{proof}

\begin{proof}[Proof of Corollary \ref{thm:finitesampleconsistency}]
Let $K,M \in \mathbb{N}$ and define
\begin{align*}
\Gamma:=\frac12\left(\frac{\operatorname{kersize}^{\mathrm a}(F,\mathcal M_1,\mathcal E,p)^p}{(F_*\mu)(\mathcal M_2)}\right)^{1/p}.
\end{align*}
By \cite[Theorem 3.9]{gottschling2023existence}, every measurable approximate inverse map $\phi:\mathcal M_2\to\mathcal X$ satisfies $\Gamma
\le \mathcal E(\phi)$. Fix $\varepsilon>\delta'>0$ and let $\delta = (2\delta')^p>0$. By Theorem~\ref{thm:convavgalgnon}, with probability greater than or equal to
$1-c(K,\delta)$,
\begin{align*}
\left|\operatorname{Kersize}(F,\mathcal M_1,\mathcal E,p)_K^p
-\frac{\operatorname{kersize}^{\mathrm a}(F,\mathcal M_1,\mathcal E,p)^p
}{(F_*\mu)(\mathcal M_2)}\right|\le\delta.
\end{align*}
Define the above event as 
\begin{align*}
C_{K,\delta}:=\left\{\left|\operatorname{Kersize}(F,\mathcal M_1,\mathcal E,p)_K^p
-\frac{\operatorname{kersize}^{\mathrm a}(F,\mathcal M_1,\mathcal E,p)^p
}{(F_*\mu)(\mathcal M_2)}\right|\le\delta\right\},
\end{align*}
and note that $\mathbb{P}(C_{K,\delta}) \geq 1-c(K,\delta)$. On the event $C_{K,\delta}$ we have
\begin{align*}
\operatorname{Kersize}(F,\mathcal M_1,\mathcal E,p)_K^p
\le
\frac{\operatorname{kersize}^{\mathrm a}(F,\mathcal M_1,\mathcal E,p)^p}{(F_*\mu)(\mathcal M_2)}+\delta.
\end{align*}
Taking $p$-th roots and using $(a+b)^{1/p}\le a^{1/p}+b^{1/p}$ yields
\begin{align*}
\Gamma_K=\tfrac12\,\operatorname{Kersize}(F,\mathcal M_1,\mathcal E,p)_K
\le
\Gamma+\tfrac12\,\delta^{1/p}.
\end{align*}
Since $\Gamma\le \mathcal E(\phi)$, it follows that on $C_{K,\delta}$, $\Gamma_K\le \mathcal E(\phi)+\tfrac12\,\delta^{1/p}$. Now define the event
\begin{align*}
A_{M,K} :=\left\{\mathcal L(\phi,\mathcal D_M,p)<\Gamma_K-\varepsilon\right\}.
\end{align*}
On the event $A_{M,K} \cap C_{K,\delta}$, we obtain $\mathcal L(\phi,\mathcal D_M,p) < \mathcal E(\phi)+\delta'-\varepsilon$, and since $\delta' = \tfrac12\,\delta^{1/p}$,
\begin{align*}
\mathcal E(\phi)-\mathcal L(\phi,\mathcal D_M,p)> \varepsilon-\delta'.
\end{align*}
Therefore,
\begin{align*}
A_{M,K}\subseteq \left(A_{M,K} \cap C_{K,\delta}\right)\bigcup A_{M,K} \cap C_{K,\delta}^C = \left\{\left|\mathcal L(\phi,\mathcal D_M,p)-\mathcal E(\phi)\right|>\varepsilon-\delta'\right\}\cup C_{K,\delta}^C,
\end{align*}
where from Theorem~\ref{thm:convavgalgnon} we have that $\mathbb P(C_{K,\delta}^C) \le c(K,\delta) = c(K,(2\delta')^p)$. Applying the union bound yields
\begin{align*}
\mathbb P(A_{M,K})\le \mathbb P\left(\left|\mathcal L(\phi,\mathcal D_M,p)
-\mathcal E(\phi)\right|>\varepsilon-\delta'\right)+c(K,(2\delta')^p).
\end{align*}
Finally, since $\left|\mathcal L(\phi,\mathcal D_M,p)-\mathcal E(\phi)\right|\to0$ in probability as $M\to\infty$, and $\tilde{c}(K,\delta')=c(K,(2\delta')^p)\to0$ as $K\to\infty$, we conclude that
\begin{align*}
\mathbb P\left(\mathcal L(\phi,\mathcal D_M,p)<\Gamma_K-\varepsilon\right)\to0.
\end{align*}
\end{proof}

\end{document}